\documentclass[a4paper,10pt]{amsart}

\usepackage[utf8]{inputenc}
\usepackage[english]{babel}
 
\usepackage{amsmath, amssymb, amsfonts, mathtools, amsthm , enumitem, hyperref, mathrsfs, standalone, color,tikz}

\usepackage[all,cmtip]{xy}

\usepackage{accents}

\newcommand{\dbtilde}[1]{\accentset{\approx}{#1}}
\newcommand{\ZZ}{\overset \longrightarrow{ Z_{\mathfrak p, \mathfrak q}^\mu}}

\newcommand{\abs}[1]{\left \lvert#1 \right \rvert}
\newcommand{\norm}[1]{\left \lVert #1 \right \rVert}

\newcommand{\T}[1]{\textup{#1}}

\newcommand{\id}{\textup{id}}
\newcommand{\ad}{\textup{ad}}
\newcommand{\diag}{\textup{diag}}

\newcommand{\e}{\varepsilon}

\newcommand{\C}{\mathbb C}

\newcommand{\N}{\mathbb N}

\newcommand{\K}{\T{K}}

\newcommand{\cG}{\mathcal G}
\newcommand{\cR}{\mathcal R}

\newcommand{\cO}{\mathcal O}

\newcommand{\cZ}{\mathcal Z}

\newcommand{\cF}{\mathcal F}

\newcommand{\cstar}{\mathrm C^\ast}
\newcommand{\Star}{{}^\ast}

\newcommand{\bb}[1]{\mathbb #1}

\newcommand{\mf}[1]{\mathfrak #1}

\swapnumbers 

\newtheorem{theorem}{Theorem}[section]
\newtheorem{lemma}[theorem]{Lemma}
\newtheorem{prop}[theorem]{Proposition}
\newtheorem{cor}[theorem]{Corollary}

\theoremstyle{definition}
\newtheorem{defn}[theorem]{Definition}
\newtheorem{example}[theorem]{Example}
\newtheorem{notation}[theorem]{Notation}
\newtheorem{convention}[theorem]{Convention}
\newtheorem{remark}[theorem]{Remark}
\newtheorem{note}[theorem]{}

\numberwithin{equation}{section}

\title{The Jiang-Su algebra is strongly self-absorbing revisited}
\author{Andr{\'e} Schemaitat}
\thanks{Supported by the Deutsche Forschungsgemeinschaft (DFG) under Germany’s Excellence Strategy - EXC 2044-390685587  Mathematics M{\"u}nster:  Dynamics - Geometry - Structure and by SFB 878.}

\begin{document}

\begin{abstract}
	We give a shorter proof of the fact that the Jiang-Su algebra is strongly self-absorbing. This is achieved by introducing and studying so-called unitarily suspended endomorphisms of generalized dimension drop algebras. Along the way we prove uniqueness and existence results for maps between dimension drop algebras and UHF-algebras.
\end{abstract}

\maketitle

\section{Introduction}

In Elliott's classification program for nuclear $\cstar$-algebras, a particularly prominent role is played by strongly self-absorbing $\cstar$-algebras (see \cite{R}, \cite{WELL}). A separable and unital $\cstar$-algebra $D \ncong \bb C$ is called \textit{strongly self-absorbing} if there exists a $\Star$-isomorphism $D \to D \otimes D$, which is approximately unitarily equivalent to the first factor embedding $\id_D \otimes 1_D$ (\cite{TW}). Strongly self-absorbing $\cstar$-algebras have approximately inner flip and hence are simple and nuclear (\cite{ER}). The list of known strongly self-absorbing $\cstar$-algebras is quite short and consists of the Cuntz-algebras $\cO_2, \cO_\infty$, UHF-algebras of infinite type, tensor products of UHF-algebras of infinite type with $\cO_\infty$ and the Jiang-Su algebra $\cZ$.

\par The Jiang-Su algebra $\cZ$ has been introduced by Jiang and Su in their remarkable paper \cite{JS}. For reasons we explain below, it is today one of the most natural and important objects in the classification theory of $\cstar$-algebras. Inspired by the work of Connes on the uniqueness of the hyperfinite II$_1$-factor $\cR$, Jiang and Su already observed that $\cZ$ is strongly self-absorbing (\cite[Proposition 8.3, Corollary 8.8]{JS}). Before we discuss their proof of this fact in more detail, let us briefly put $\cZ$ into more context.

\par For the classification of $\cstar$-algebras, by what is today known as the Elliott invariant, $\cZ$-stability is of fundamental importance. The Elliott invariant of a unital $\cstar$-algebra $A$ is given by the six-tupel 
$$
	(\K_0(A),\K_0(A)_+,[1_A]_0, \K_1(A), T(A), r_A).
$$
The map $r_A$ is a natural pairing between K-theory and the trace simplex $T(A)$. Since $\cZ$ has a unique trace and is KK-equivalent to the complex numbers $\bb C$, in fact it has the same Elliott invariant as $\bb C$, one might think of $\cZ$ as an infinite dimensional version of the complex numbers.  Under a natural restriction on the K-groups of a simple $\cstar$-algebra $A$, the Elliott invariant of $A$ and $A \otimes \cZ$ are isomorphic. This suggests that classification is in general only possible up to $\cZ$-stability (cf.~\cite[Definition 2.6]{WELL}).

\par Intimately related is the Toms-Winter conjecture (cf.~\cite[Conjecture 9.3]{WZDIMNUC}), which predicts that the following  regularity conditions for a (unital) separable, simple,  nuclear and infinite dimensional $\cstar$-algebra $A$ are equivalent:
\begin{enumerate}[label=(\roman*)]
	\item $A$ has finite nuclear dimension,
	\item $A$ is $\cZ$-stable (i.e. $A \cong A \otimes \cZ$),
	\item $A$ has strict comparison.
\end{enumerate} 
The implication (ii) $\Rightarrow$ (iii) is proven in  \cite[Corollary 4.6]{RSTABLE}, where the strong self-absorption of $\cZ$ is used in an essential way. Just recently, in \cite{CETWW}, the nuclear dimension of a $\cstar$-algebra like $A$ has been computed to be zero or one in the $\cZ$-stable case and infinite otherwise. Together with a result by Winter (\cite[Corollary 6.3]{WPURE}) this proves the equivalence between (ii) and (i). By previous results of many hands and decades of work, one of the most outstanding results in  the classification of $\cstar$-algebras follows: unital, separable, simple, nuclear and $\cZ$-stable $\cstar$-algebras satisfying the UCT are classified by the Elliott invariant (\cite[Corollary D]{CETWW}). In particular unital, separable, simple and nuclear $\cstar$-algebras satisfying the UCT are classified up to $\cZ$-stability. The UCT is Rosenberg and Schochet's universal coefficient theorem (\cite{UCT}), which roughly says that the  KK-theory can be computed in terms of K-theory. A $\cstar$-algebra satisfies the UCT precisely if it is KK-equivalent to a commutative $\cstar$-algebra.

\par Getting back to the strong self-absorption of $\cZ$, let us discuss the proof of Jiang and Su. In \cite[Section 2]{JS} the algebra $\cZ$ is constructed as an inductive limit of dimension drop algebras $Z_{p_n,q_n}$ with $p_n,q_n$ coprime and tending to infinity, where $Z_{p,q}$ is the $\cstar$-algebra of continuous functions $f \colon [0,1] \to M_p \otimes M_q$ such that $f(0) \in M_p \otimes 1_q$ and $f(1) \in 1_p \otimes M_q$. Fundamental to their proof that $\cZ$ is strongly self-absorbing  is a classification machinery for $\Star$-homomorphisms between dimension drop algebras (\cite[Corollary 5.6, Theorem 6.2]{JS}). These results rely on a careful analysis of maps between K-homology induced by morphisms between dimension drop algebras. 
\par Although the notion of a strongly self-absorbing $\cstar$-algebra was introduced first in \cite{TW}, Jiang and Su already proved two abstract properties of $\cZ$, showing that $\cZ$ is strongly self-absorbing:
\begin{enumerate}
	\item $\cZ$ has approximately inner half flip, i.e.~the first and the second factor embedding of $\cZ$ into $\cZ \otimes \cZ$ are approximately unitarily equivalent,
	\item there exists a $\Star$-homomorphism $\varphi \colon \cZ \otimes \cZ \to \cZ$ such that 
	$\varphi \circ (\id_\cZ \otimes 1_\cZ)$ is approximately inner.
\end{enumerate}
The first property is proven in \cite[Proposition 8.3]{JS}, only using  basic properties of the construction of $\cZ$. The second property is proven in \cite[Proposition 8.5]{JS} and relies heavily on the classification machinery just mentioned. Jiang and Su first show that there exists a $\Star$-homomorphism $\psi \colon \cZ \otimes \cZ \to B$, where $B$ is a simple and unital inductive limit of dimension drop algebras having the same Elliott invariant as $\cZ$. Now their existence result (\cite[Theorem 6.2]{JS}) provides a unital $\Star$-homomorphism $B \to \cZ$. By composing with $\psi$ one gets the desired map $\varphi$. Using that any unital $\Star$-endomorphism of $\cZ$ is approximately inner (\cite[Theorem 7.6]{JS}), it follows that $\varphi \circ (\id_\cZ \otimes 1_\cZ)$ is approximately inner. Here  the  uniqueness result for maps between dimension drop algebras is used once more.

\par The goal of this paper is to prove that the Jiang-Su algebra is strongly self-absorbing in an as self-contained and elementary  as possible way. To do so, we use a different picture of $\cZ$. Today, there are many descriptions and characterizations of $\cZ$, for example as universal $\cstar$-algebra (\cite{JW}) or as the initial object in the category of strongly self-absorbing $\cstar$-algebras (\cite{WSSA}). For us  however,  a construction of $\cZ$ as an inductive limit of generalized dimension drop algebras will be most suitable. A generalized dimension drop algebra $Z_{\mf p, \mf q}$ is defined just as before, but now with $M_p$ and $M_q$ replaced by UHF-algebras $M_\mf p$ respectively $M_\mf q$, associated to supernatural numbers $\mf p$ and $\mf q$. More precisely,  R{\o}rdam and Winter show in   \cite[Theorem 3.4]{RW} that for all coprime supernatural numbers of infinite type $\mf p$ and $\mf q$, there exists a trace collapsing\footnote{Trace collapsing means that $\tau \circ \mu = \tau' \circ \mu$, for all tracial states $\tau$ and $\tau'$ on $Z_{\mf p, \mf q}$.} $\Star$-homomorphism $\mu \colon Z_{\mf p, \mf q} \to Z_{\mf p, \mf q}$ and that the stationary inductive limit of $Z_{\mf p, \mf q}$ along $\mu$ is isomorphic to $\cZ$, i.e. 
\[
	\xymatrix{
	 Z_{\mf p, \mf q}  \ar[r]^-\mu &  Z_{\mf p, \mf q} \ar[r]^-\mu &  Z_{\mf p, \mf q}  \ar[r]^-\mu &  \cdots \ar[r] &  \cZ
	}.
\]
For classification, this picture of $\cZ$ has already been proven to be very useful, most notably in  \cite{WELL}. Also the main theorem of \cite[Theorem 29.5]{GLN} relies on this picture of $\cZ$. However, showing that the limit of $Z_{\mf p, \mf q}$ along $\mu$  is strongly self-absorbing still requires comparing it to the original construction of Jiang and Su.

\par Using the picture of R{\o}rdam and Winter and ideas of Winter (cf.~\cite[Section 4, Definition 4.2]{WELL}),  we introduce so-called \textit{unitarily suspended} $\Star$-endomorphisms of generalized dimension drop algebras (Definition \ref{defn:unitarily-suspeded}). Built into the definition is a unitary path, which in some sense untwists the trace collapsing $\Star$-endomorphism. This  allows to run an approximate intertwining that produces an isomorphism $\cZ \to \cZ \otimes \cZ$, which is approximately unitarily equivalent to the first factor embedding. Showing that such  unitarily suspended $\Star$-endomorphisms exist and running the approximate intertwining then provides a new way to show that $\cZ$ is strongly self-absorbing (Theorem \ref{th:SSA}, \ref{thm:structure}). This  also confirms that the difficulty of showing that $\cZ$ is strongly self-absorbing lies between that for UHF-algebras of infinite type and $\mathcal O_2$ or $\mathcal O_\infty$ (cf.~\cite[{}5.2]{JW}).

	\par Finally, I want to highlight that in Theorem \ref{th:SSA} no classification theory beyond the classification of UHF-algebras is used. In particular, we do not rely on the classification of morphisms between dimension drop algebras or K-homological results. This is possible by proving existence and uniqueness results for the very specific class of maps between generalized prime dimension drop algebras and certain UHF-algebras.
	
\section*{Acknowledgment}

This work is part of my doctoral research. I thank my supervisor Wilhelm Winter for suggesting the topic and for the many enlightening discussions. Furthermore, I thank James Gabe, Christopher Schafhauser and especially Stuart White for their very helpful feedback.

\section{Preliminaries}

In this section we recall some basic facts about c.p.c. order zero maps, generalized dimension drop algebras, strongly self-absorbing $\cstar$-algebras and ultrapowers. See \cite{Order-Zero}, \cite{RW} and \cite{TW} for more.

\subsection{Order zero maps}
	\label{def:cpc-order-zero}
	Let $A$ and $B$ be $\cstar$-algebras.  We say that a linear and completely positive contractive (c.p.c.) map $\varphi \colon A \to B$ is \textit{order zero} if $\varphi(x)\varphi(y) = 0$, whenever $xy = 0$. In this case, there exists a $\Star$-homomorphism  $\pi_\varphi \colon A \to B^{**}$, called the \textit{supporting $\Star$-homomorphism}, such that $\varphi(\cdot) = h_\varphi \pi_\varphi(\cdot) = \pi_\varphi(\cdot)h_\varphi$, where $h_\varphi$ is a positive contraction in the double dual $B^{**}$. If $A$ is unital, then $h_\varphi = \varphi(1_A)$.  Moreover, a c.p.c. order zero map $\varphi \colon A \to B$  corresponds bijectively to a $\Star$-homomorphism $\overline \varphi \colon C_0(0,1] \otimes A \to B$, where $\overline \varphi(\iota \otimes x) = \varphi(x)$, for $x \in A$. By $\iota$ we denote the canonical generator of $C_0(0,1]$. Similarly, if $\phi \colon C_0(0,1] \otimes A \to B$ is a $\Star$-homomorphism, we denote by $\overline{\phi} \colon A \to B$ the associated c.p.c. order zero map, given by $\overline \phi (a) = \phi(\iota \otimes a)$.

\subsection{Generalized dimension drop algebras}
	\label{def:dimension-drop}
	For (possibly finite) supernatural numbers $\mf p$ and $\mf q$ we define
$$
		Z_{\mf p, \mf q} \coloneqq \left \{f \in C([0,1], M_\mathfrak p \otimes M_\mathfrak q) : \begin{array}{l} 	
			f(0) \in M_\mathfrak p \otimes 1_\mathfrak q , \\
			f(1) \in 1_\mathfrak p \otimes M_\mathfrak q 
			\end{array}
			\right  \}.
	$$ 
	By $M_\mf p$ and $M_\mf q$ we denote the UHF-algebras associated to the respective supernatural numbers. We call $Z_{\mf p, \mf q}$ a \textit{generalized dimension drop} algebra. Note that $Z_{\mf p, \mf q}$ is canonically generated by  two commuting cones over $M_{\mf p}$ respectively $M_{\mf q}$, which we will denote by 
	$$
		\grave \imath \colon  C_0[0,1) \otimes M_\mf p \hookrightarrow Z_{\mf p, \mf q} \quad  \text{and} \quad 
		\acute \imath \colon C_0(0,1] \otimes M_\mf q \hookrightarrow Z_{\mf p, \mf q}.
$$
	 Whenever  $\varphi \colon Z_{\mf p, \mf q} \to A$ is a $\Star$-homomorphism, we write
	$$
		\grave \varphi \coloneqq  \varphi \circ \grave \iota  \quad \text{and} \quad \acute \varphi \coloneqq \varphi \circ \acute \iota.
	$$
	A unital $\Star$-homomorphism $\varphi \colon Z_{\mf p, \mf q} \to A$ is called \textit{standard}, if $\tau \circ \varphi = \tau_{\T{Leb}}$, for all $\tau \in T(A)$ (cf.\,\cite[Theorem 2.1]{RSTABLE}). Here $\tau_{\T{Leb}}$ is the trace on $Z_{\mf p, \mf q}$ induced by the Lebesgue measure.
\par Furthermore, if $\psi \colon Z_{\mf p , \mf q} \to Z_{\mf p, \mf q}$ is a $\Star$-homomorphism we denote for $s \in [0,1]$ the map
$$
	\psi_s \colon Z_{\mf p, \mf q} \to M_\mf p \otimes M_\mf q : f \mapsto \psi(f)(s).
$$

\subsection{Strongly self-absorbing $\cstar$-algebras.}
	\label{def:ssa}
 	A unital and separable $\cstar$-algebra $D \ncong \C$ is called \textit{strongly self-absorbing} if there exists a $\Star$-isomorphism $\varphi \colon D \to D \otimes D$ and a sequence of unitaries $(u_n)_{n=1}^\infty$ in $ D \otimes D$, such that 
 	$$
 		\varphi(x) = \lim_{n \to \infty} u_n(x \otimes 1_D)u_n^* \qquad ( x \in D) .
 	$$	
 	\par The most basic examples of strongly self-absorbing $\cstar$-algebras are UHF-algebras of infinite type, i.e. those UHF-algebras $M_\mf  n$, such that $\mf n^2 = \mf n$ and $\mf n \neq 1$. In this case the proof of strong self-absorption reduces to basic linear algebra.
 	
\subsection{Ultrapowers}
Let $\omega$ be a fixed free ultrafilter on $\N$. We denote the \textit{uniform ultrapower} of $A$ by 
$$
 A_\omega \coloneqq \ell^\infty(A) / \{ (x_n)_{n=1}^\infty : \lim_{n \to \omega} \norm {x_n} = 0 \}.
$$  
 The equivalence class of a sequence $(x_n)_{n =1}^\infty \in \ell^\infty(A)$ will be denoted by $\big [ (x_n)_{n=1}^\infty \big ] \in A_\omega$.  By $\iota_\omega$ we denote the diagonal inclusion $A \hookrightarrow A_\omega$. Often it will be  convenient to identify $A$ in this way  with a subalgebra of $A_\omega$.

\section{Uniqueness}

The main goal of this section is to prove an asymptotic uniqueness result for standard $\Star$-homomorphisms from $Z_{\mf p, \mf q }$ into $M_\mf p \otimes M_\mf q$ (Theorem \ref{th:uniqueness}). This is done by first proving an approximate uniqueness result for these maps (Lemma \ref{lem:aue-dimension-drop}) and a so-called Basic Homotopy Lemma (Lemma \ref{lem:BHL-dimension-drop}). The key  ingredients to this section are a version of the Basic Homotopy Lemma for positive contractions (\cite[Lemma 5.1]{BEEK}) and certain unitaries associated to order zero maps, which allow to twist non-central elements of the domain into an extra tensor factor of the codomain (see \ref{rotation-unitaries}).

\begin{defn}
	\label{def:equivalences}
	Let $A$ and $B$ be $\cstar$-algebras with $A$ separable and $B$ unital. For $\Star$-homomorphisms $\varphi, \psi \colon A \to B$ we define the following equivalence relations:
	\begin{enumerate}[label=(\roman*)]
		\item \textit{Unitary equivalence} $\sim_{\T u}$ : There exists a unitary $u \in B$ with 
		$$
			u\varphi(a)u^* = \psi(a) \qquad (a \in A).
		$$
		\item \textit{Approximate unitary equivalence} $\approx_{\T u}$ : There exists a sequence $(u_n)_{n=1}^\infty$ of unitaries in $B$ with
		 $$
			\lim_{n \to \infty} \norm{u_n \varphi(a) u_n^* - \psi(a)} = 0 \qquad (a \in A).		
		$$
		 \item \textit{Strong asymptotic unitary equivalence} $\sim_{\T{asu}}$ : There exists a unitary path $(u_t)_{t \in [0,1)}$ in $B$ such that $u_0 = 1_B$ and
		$$
			\lim_{t \to 1} \norm{u_t \varphi(a) u_t^* - \psi(a)}  = 0 \qquad (a \in A).
		$$ 
		\item \textit{Murray-von Neumann equivalence} $\sim$ :  There exists a contraction $v \in B$ with 
		$$
			v\varphi(a)v^* = \psi(a),  \quad  v^*\psi(a)v = \varphi(a) \qquad (a \in A).
		$$
		\item \textit{Approximate Murray-von Neumann equivalence} $\approx$ : There exists a sequence $(v_n)_{n=1}^\infty$ of contractions in $B$ with 
		\begin{align*}
			\lim_{n \to \infty} \norm{v_n \varphi(a)v_n^* - \psi(a)} = 0  \qquad (a \in A), \\
			 \lim_{n \to \infty} \norm{v_n^*\psi(a)v_n - \varphi(a)} = 0 \qquad (a \in A).
		\end{align*}
	\end{enumerate}
	These equivalence relations also make sense for c.p.c. order zero maps, as they are identified with $\Star$-homomorphisms on the corresponding cone (see \ref{def:cpc-order-zero}). The same applies to positive contractions, which correspond to c.p.c. order zero maps on $\bb C$.
\end{defn}

\begin{remark}
	\label{rem:ultrapower}
	 (\cite[Lemma 6.2.5]{R}) Let $A$ and $B$ be $\cstar$-algebras with $A$ separable and $B$ unital. For $\Star$-homomorphisms $\varphi,\psi \colon A \to B$ we have 
	$$
		\varphi \approx_{\T u} \psi  \iff  \iota_\omega \circ \varphi \sim_{\T u} \iota_\omega \circ \psi \iff \iota_\omega \circ \varphi \approx_{\T u} \iota_\omega \circ \psi.
	$$
	Furthermore, if $\varphi, \psi \colon A \to B_\omega$ are $\Star$-homomorphisms, a reindexing argument shows that 
	$$
		\varphi \approx_{\T u} \psi \quad  \iff \quad \varphi \sim_{\T u} \psi.
	$$
\end{remark}

The next lemma is inspired by the classification of $\Star$-homomorphism from finite dimensional $\cstar$-algebras into certain unital $\cstar$-algebras (cf.~\cite[Lemma 1.3.1]{R}).

\begin{lemma}
	\label{lem:cpc-order-zero-e11}
	Let $A$ be a unital $\cstar$-algebra, $p \in \N$ and assume that  $\varphi, \psi \colon M_p \to A$ are c.p.c. order zero maps, such that $\varphi(e_{11}) \approx \psi(e_{11})$. Then $\varphi \approx \psi$. If in addition $A$ has stable rank one, i.e. the set of invertible elements is dense in $A$, then $\varphi \approx_{\T u} \psi$.
\end{lemma}

\begin{proof}
	Assume $\varphi(e_{11}) \approx  \psi(e_{11})$. By a standard argument,  there exists a contraction $v \in A_\omega$ with 
	$$
		v\varphi(e_{11})v^* = \psi(e_{11}), \quad v^*\psi(e_{11})v = \varphi(e_{11}).
	$$
	Consider the elements  $e_\varphi \coloneqq  \big [ (h_\varphi^{\frac 1 n})_{n=1}^\infty \big ] $ and $e_\psi \coloneqq \big  [(h_\psi^{\frac 1 n})_{n=1}^\infty \big ]$ in $A_\omega$, where $h_\varphi$ and $h_\psi$ are as in \ref{def:cpc-order-zero}. The elements $e_\varphi$ and $e_\psi$ act as unit on the image of $\varphi$ respectively $\psi$. Considering $A_\omega \subseteq (A^{**})_\omega$, we furthermore have that $e_\varphi \pi_\varphi (x)$ and $e_\psi \pi_\psi(x)$ are elements in $A_\omega$, for every $x \in M_p$. One then defines 
	$$
		u \coloneqq \sum_{i=1}^p \pi_\psi(e_{i1}) e_\psi v e_\varphi \pi_\varphi(e_{1i}) \in A_\omega
	$$ 
	and computes that 
	$$
		u \varphi(e_{ij}) u^* = \psi(e_{ij}), \quad u^*\psi(e_{ij})u = \varphi(e_{ij}) \qquad ( i,j = 1,2, \cdots, p).
	$$	
	This entails that $\varphi \approx \psi$. The last statement follows from \cite[Proposition 3.13]{G}.
\end{proof}

\begin{defn}
	\label{def:Lebesgue-Contraction}
	Let $A$ be unital $\cstar$-algebra with $T(A) \neq \emptyset$. A positive contraction $h$ in $A$ is called a \textit{Lebesgue contraction} if 
	$$
		\tau(f(h)) = \int_0^1 f(t) \ dt \qquad (f \in C_0(0,1], \ \tau \in T(A)).
	$$	
\end{defn}

By appealing to more sophisticated tools like the Cuntz semigroup, it  follows for example from \cite[Lemma 2.1]{TWW} that the  next lemma holds for a more general codomain, including ultrapowers of UHF-algebras. However, the proof here uses \cite{SK}, which provides a self-contained proof for the case of UHF-algebras.

\begin{lemma}
		\label{lem:!-cpc-order-zero}
		Let $\mf p$ be a supernatural number and let $\varphi, \psi \colon M_\mf p \to U$ be two c.p.c. order zero maps, where $U$ is a $\T{UHF}$-algebra. Assume that $h_\varphi$ and $h_\psi$ are both Lebesgue contractions. Then $\varphi \approx_{\T u} \psi$.
\end{lemma}

\begin{proof}
	By an easy approximation argument one may assume that  $\varphi, \psi \colon M_p \to U$ are two standard c.p.c. order zero maps, where $p \in \N$. Note that the positive contractions $\{\varphi(e_{ii})\}_{i=1}^p$ are mutually orthogonal and equivalent positive elements, in the sense that there exist $x_i \in U$ with $\varphi(e_{11}) = x_i^*x_i$ and $x_ix_i^* = \varphi(e_{ii})$. Denote by $\tau_U$ the unique trace on $U$. Using that the $\varphi(e_{ii})$ are mutually equivalent, it is an easy exercise to show that $\tau_U(f(\varphi(e_{11}))) = \tau_U(f(\varphi(e_{ii})))$, for $i=1,2, \cdots, p$ and $f \in C_0(0,1]$. Using that $h_\varphi = \varphi(1_p)$ is a Lebesgue contraction, it follows that
	\begin{align*}
		\frac 1 p \int_0^1 f(t) \ dt &  = \frac 1 p\  \tau_U(f(\varphi(1_p))) = \frac 1 p \sum_{i=1}^p \tau_U(f(\varphi(e_{ii})))  \\
		& = \frac 1 p \  \sum_{i=1}^p \tau_U(f(\varphi(e_{11}))) = \tau_U(f(\varphi(e_{11}))) ,
	\end{align*}
	for all $f \in C_0(0,1]$. The same computation works for $\psi$ and we see that both $\varphi(e_{11})$ and $\psi(e_{11})$ have $\frac 1 p$ Lebesgue spectral measure. In particular they have the same eigenvalue function, as defined in \cite[Definition 2.6]{SK}. Since $U$ has real rank zero and strong comparison of projections\footnote{Meaning that two projections in $U$ are equivalent if they have the same trace.} by the unique trace $\tau_U$, it follows by \cite[Theorem 5.1]{SK} that $\varphi(e_{11}) \approx_{\T u} \psi(e_{11})$ and in particular $\varphi(e_{11}) \approx \psi(e_{11})$. Since $U$ has stable rank one, Lemma  \ref{lem:cpc-order-zero-e11} shows that $\varphi \approx_{\T u} \psi$.
\end{proof}

\begin{note}
	\label{rotation-unitaries}
	\textbf{Rotation trick.} The following provides a method that allows to reduce many arguments about maps on $Z_{\mf p, \mf q}$ to maps on $C([0,1])$. This ingredient is essential for the proofs of Lemma \ref{lem:aue-dimension-drop} and Lemma \ref{lem:BHL-dimension-drop}.
	\par   Let $\varphi \colon Z_{\mf p, \mf q} \to U$ be a unital $\Star$-homomorphism, where $U$ is some $\cstar$-algebra (in the following a $\T{UHF}$-algebra or an ultrapower of such). Let $\mf p$ and $\mf q$ be  supernatural numbers of infinite type.	By \cite[Theorem 2.2]{DW}, the flip map on $M_\mf p \otimes M_\mf p$ and $M_\mf q \otimes M_\mf q$ is strongly asymptotically inner, witnessed by say unitary paths $S_\mf p \colon [0,1) \to M_\mf p \otimes M_\mf p$ and  $S_\mf q \colon [0,1) \to M_\mf q \otimes M_\mf q$. For $t \in [0 ,1)$, we define a continuous function
	\[
		 \xi_t \colon [0,1] \to [0,1) : s \mapsto
		 \begin{cases}	
		 	t & \text{ if } s \in [0,t], \\
		 	\text{linear} & \text{ if } s \in [t,1], \\
		 	0 & \text{ if } s = 1
		 \end{cases}
	\]
	and unitary paths 
	\begin{eqnarray}
		\label{eq:W-def}
		\begin{aligned}
		\acute W_t &   \colon [0,1] \to M_{\mf q} \otimes M_{\mf q}  : s \mapsto S_{\mf q}( \xi_t(1-s)) , \\
		\grave W_t & \colon [0,1] \to  M_{\mf p} \otimes M_{\mf p} : s \mapsto S_\mf p( \xi_t(s)).
	\end{aligned}
	\end{eqnarray}
	The path $\acute W_t$, for example, is designed in such a way that it takes most of the time, namely on $[1-t,1]$,  the value $S_\mf q(t)$, which implements the flip better and better. Furthermore $\acute W_t(0)= 1_\mf q \otimes 1_\mf q$, showing that $\acute W_t$ defines an element of the minimal unitization $(C_0(0,1] \otimes M_\mf q \otimes M_\mf q)^+$ and $\acute W_t - 1 \in C_0(0,1] \otimes M_\mf q \otimes M_\mf q$, where $1$ denotes the adjoined unit. Note that we freely identify $C_0(0,1] \otimes M_\mf q \otimes M_\mf q$ with the $\cstar$-algebra of continuous functions $\{f \colon [0,1] \to M_\mf q \otimes M_\mf q: f(0) = 0\}$. Similarly, $\grave W_t \in (C_0[0,1) \otimes M_\mf p \otimes M_\mf p)^+$ with scalar part $1$. One then easily checks that 
	\begin{eqnarray}
		\label{eq:W-equation}
		\begin{aligned}
		\lim_{t \to 1} \ &  \acute W_t (g \otimes r \otimes s) (\acute W_t)^*  = g \otimes s \otimes r , \\ 
		\lim_{t \to 1} \  & \grave W_t (f \otimes x \otimes y) (\grave W_t)^*  = f \otimes y \otimes x ,
		\end{aligned}
	\end{eqnarray}
	for $g \in C_0(0,1], \ r,s \in M_\mf q$ and $f \in C_0[0,1), \ x,y \in M_\mf p$. Since the paths $\acute W_t$ and $\grave W_t$ depend continuously on the parameter $t$, we get unitary paths $(\grave U_t^\varphi)_{t \in [0,1)}$ and $(\acute U_t^\varphi)_{t \in [0,1)}$ defined by
	\begin{equation}
		\label{eq:U-def}
		\begin{aligned}
		\grave U^\varphi_t &  \coloneqq (\grave \varphi \otimes \id_{M_\mf p})^+(\grave W_t) \in U \otimes M_\mf p, \\
		\acute U^\varphi_t & \coloneqq (\acute \varphi \otimes \id_{M_\mf q})^+ ( \acute W_t ) \in U \otimes M_\mf q.
		\end{aligned}
	\end{equation}
 	See \ref{def:dimension-drop} for the definition of $\acute \varphi$ and $\grave \varphi$. From \eqref{eq:W-equation}, it follows that we may twist the non-central part of  each cone separately into an extra tensor factor (as done in Lemma \ref{lem:aue-dimension-drop}):
\begin{eqnarray}
	\label{eq:U-equation}
	\begin{aligned}
	\lim_{t \to 1} \ &  \acute  U^\varphi_t (\acute  \varphi(g \otimes r) \otimes 1_\mf q) (\acute  U^\varphi_t)^* = \acute  \varphi(g \otimes 1_\mf q) \otimes r = \varphi(g) \otimes r , \\
	\lim_{t \to 1}\  & \grave  U^\varphi_t (\grave  \varphi(f \otimes x) \otimes 1_\mf p) (\grave  U^\varphi_t)^* = \grave  \varphi(f \otimes 1_\mf p) \otimes x = \varphi(f) \otimes x.
	\end{aligned}
	\end{eqnarray}
\par \noindent If we use $U \otimes M_\mf p \otimes M_\mf q$ as codomain, as in Lemma \ref{lem:BHL-dimension-drop},  we instead define  
\begin{equation}
		\label{eq:U-def-pq}
		\begin{aligned}
		\grave U^\varphi_t &  \coloneqq (\grave \varphi \otimes \imath^{[1]}_{(\mf p, \mf q)})^+(\grave W_t) \in U \otimes M_\mf p \otimes M_\mf q, \\
		\acute U^\varphi_t & \coloneqq (\acute \varphi \otimes \imath_{(\mf p, \mf q)}^{[2]})^+ ( \acute W_t ) \in U \otimes M_\mf p \otimes M_\mf q,
		\end{aligned}
	\end{equation}
where $\imath_{(\mf p, \mf q)}^{[1]}$ and $\imath_{(\mf p, \mf q)}^{[2]}$ denote the obvious embeddings of $M_\mf p$ resp. $M_\mf q$ into the first resp. second tensor factor of $M_\mf p \otimes M_\mf q$ (see Notation  \ref{notation:factor-embeddings}). In that case
\begin{equation}
	\label{eq:U-commutes}
	 [\grave U_t^\varphi, \acute U_t^\varphi] = 0 \qquad (t \in [0,1)).
\end{equation}
Then, the obvious adjustment of \eqref{eq:U-equation} holds:
\begin{eqnarray*}
	\lim_{t \to 1} &  \acute  U^\varphi_t (\acute  \varphi(g \otimes r) \otimes 1_{\mf p \mf q}) (\acute  U^\varphi_t)^* = \acute  \varphi(g \otimes 1_\mf q) \otimes 1_\mf p \otimes r = \varphi(g) \otimes 1_\mf p \otimes r,
	\\ 
	\lim_{t \to 1} & \grave U_t^\varphi (\grave \varphi(f \otimes x) \otimes 1_{\mf p \mf q}) (\grave U_t^\varphi)^* = \grave \varphi(f \otimes 1_\mf p) \otimes x \otimes 1_\mf q = \varphi(f) \otimes x \otimes 1_\mf q,
\end{eqnarray*}
where $1_{\mf p \mf q} \coloneqq 1_\mf p \otimes 1_\mf q$. Most importantly, in this situation we may twist the non-central part of both cones simultaneously into commuting subalgebras of $U \otimes M_\mf p \otimes M_\mf q$:
\begin{eqnarray}
	\label{eq:twist}
	\begin{aligned}
	\lim_{t \to 1} \  & \grave U_t^\varphi \acute U_t^\varphi ( \grave \varphi(f \otimes x) \otimes 1_{\mf p \mf q}) (\acute U_t^\varphi)^* (\grave U_t^\varphi)^* = \varphi(f) \otimes x \otimes 1_\mf q, \\
	\lim_{t \to 1} \  &  \grave U_t^\varphi \acute U_t^\varphi ( \acute \varphi(g \otimes y) \otimes 1_{\mf p \mf q}) (\acute U_t^\varphi)^* (\grave U_t^\varphi)^* = \varphi(g) \otimes 1_\mf p \otimes  y,
		\end{aligned}
\end{eqnarray}
for all $f \in C_0[0,1), \ x \in M_\mf p$ and $g \in C_0(0,1], \ y \in M_\mf q$. Note that \eqref{eq:twist}  follows directly  from \eqref{eq:U-equation} and \eqref{eq:U-commutes}.
	\end{note}

\begin{lemma}
	\label{lem:aue-dimension-drop}
	Let $\mf p$ and $\mf q$ be supernatural numbers and let $U$ be a $\T{UHF}$-algebra of infinite type, which tensorially absorbs $M_\mf p$ or $M_\mf q$. Assume that $\beta_0 , \beta_1 \colon Z_{\mf p, \mf q} \to U$  are standard $\Star$-homomorphisms. Then $\beta_0 \approx_{\T u} \beta_1$.
\end{lemma}	

\begin{proof}
	Let us assume that $U$ tensorially absorbs $M_\mf p$. The other case is then obtained by reversing the orientation of the unit  interval. By assumption, 
	$$
	h_{\overline{\acute \beta_0}} \quad \text{and} \quad h_{\overline{\acute \beta_1}}
$$ are both Lebesgue contractions and by Lemma \ref{lem:!-cpc-order-zero} we see that $\acute \beta_0 \approx_{\T u} \acute \beta_1$. By Remark \ref{rem:ultrapower}, there exists a unitary $u \in U_\omega$, such that  $\ad(u) \circ  \acute \beta_0 = \acute \beta_1$, when $\acute \beta_0$ and $\acute \beta_1$ are  considered as maps taking values in $U_\omega$. Let us define $\gamma_i \colon Z_{\mf p, \mf q} \to U_\omega$ by
	$$
		\gamma_0 \coloneqq \ad(u) \circ \beta_0, \qquad \gamma_1 \coloneqq \beta_1.
	$$
	Then $\acute \gamma_0 = \acute \gamma_1$ and one easily checks that 
	\begin{equation}
		\label{eq:central}
				\gamma_0(f) = \gamma_1(f) \qquad (f \in C([0,1])).
		\end{equation}
		For $t \in [0,1)$, we define  $V_t \coloneqq (  \grave U^{\gamma_1}_t )^*\grave U^{\gamma_0}_t \in U_\omega \otimes M_\mf p \subseteq (U \otimes M_\mf p)_\omega$, with $\grave U_t^{\gamma_i}$ as in (\ref{eq:U-def}). Let $f \in C_0[0,1), \ x \in M_\mf p$ and $\e > 0$. Using (\ref{eq:U-equation}), choose $t_0 \in [0,1)$ such that 
		\begin{equation}
			\label{eq:estimate-aue}
			\grave U_t^{\gamma_i} (\grave \gamma_i (f \otimes x) \otimes 1_\mf p) (\grave U_t^{\gamma_i})^* \approx_{\frac \e 2} \gamma_i(f) \otimes x \qquad (i = 0,1, \ t \geq t_0 ).
		\end{equation}
		For $t \geq t_0$, we then get
	\begin{align*}
		V_{t}(\grave \gamma_0(f \otimes x) \otimes 1_\mf p) V_t^* & \overset{\phantom{(2.9)}} =  (  \grave U^{\gamma_1}_{t} )^*\grave U^{\gamma_0}_{t}  (\grave \gamma_0(f \otimes x) \otimes 1_\mf p)  (\grave U_{t}^{\gamma_0})^* \grave U_{t}^{\gamma_1}  \\
		& \overset{(\ref{eq:estimate-aue})} \approx_{\frac \e 2}  (\grave U_{t}^{\gamma_1})^*  (\gamma_0(f) \otimes x) \grave U_{t}^{\gamma_1} \\
		&   \overset{(\ref{eq:central})} = (\grave U_{t}^{\gamma_1})^*  (\gamma_1(f) \otimes x) \grave U_{t}^{\gamma_1}  \\
		& \overset{(\ref{eq:estimate-aue})} \approx_{\frac \e 2} \grave \gamma_1 (f \otimes x) \otimes 1_\mf p.
	\end{align*}
	Furthermore, $V_{t}$ commutes with $\acute \gamma_0 \otimes 1_\mf p = \acute \gamma_1 \otimes 1_\mf p$. It follows that 
	$$
		\lim_{n \to \infty} \norm{V_{1-\frac 1 n}(\gamma_0(g) \otimes 1_\mf p)V_{1-\frac 1 n}^* - \gamma_1(g) \otimes 1_\mf p} = 0 \qquad (g \in Z_{\mf p, \mf q}),
	$$		
	and hence  $\gamma_0 \otimes 1_\mf p \approx_{\T u} \gamma_1 \otimes 1_\mf p$. By Remark \ref{rem:ultrapower}, we  have $\gamma_0 \otimes 1_\mf p \sim_{\T u} \gamma_1 \otimes 1_\mf p$. Let  $u' \in (U \otimes M_\mf p)_\omega$, such that 
	$$
		\ad(u') \circ  ( \gamma_0 \otimes 1_\mf p) =  \gamma_1 \otimes 1_\mf p.	
	$$
	Putting $u'' \coloneqq u' (u \otimes 1_\mf p) \in (U \otimes M_\mf p)_\omega$, we thus have $\ad(u'') \circ (\beta_0  \otimes 1_\mf p) = \beta_1 \otimes 1_\mf p$. By \cite[Theorem 2.3]{TW} and the assumption that $U \cong U \otimes M_\mf p$,  there exists a $\Star$-homomorphism $\varphi \colon U \otimes M_\mf p \to U_\omega$	with $\varphi(x \otimes 1_\mf p) = x$, for $x \in U$. This induces a map 
		$$
			\varphi_\omega \colon (U \otimes M_\mf p)_\omega \to (U_\omega)_\omega.
		$$  
		Let $w \coloneqq \varphi_\omega(u'') \in (U_\omega)_\omega$. Then 
	$$
		w \beta_0(x) w^*  = \varphi_\omega(u'') \varphi_\omega(\beta_0(x) \otimes 1_\mf p) \varphi_\omega((u'')^*) = \varphi_\omega(\beta_1(x) \otimes 1_\mf p)   = \beta_1(x).
	$$
	 It follows that $\beta_0 \sim_{\T u} \beta_1$ in $(U_\omega)_\omega$. By Remark \ref{rem:ultrapower} again,  $\beta_0 \approx_{\T u}  \beta_1$ as maps into $U$.
\end{proof}

The following lemma is a slight reformulation of \cite[Lemma 5.1]{BEEK}, which is therein referred to as \textit{Isospectral Homotopy Lemma in Case of a Large Spectral Gap}.  In modern work this type of lemma is often called a \textit{Basic Homotopy Lemma} and due to its importance to the classification of $\cstar$-algebras great effort has been made to generalize the theory of such (cf.~\cite{LIN}). For our purpose however, this elementary result is sufficient. \ \\
\par Let us first recall some notation from \cite{BEEK}.
For $N \in \N$ and a positive contraction $h$ in a $\cstar$-algebra $A$, we denote 
	$$
		h(N) = \frac 1 N \sum_{k=1}^N \chi_{I_k}(h),
	$$
	where $I_k = [\frac {N-k+1}  N,1]$ and $\chi_{I_k}(h)$ is the associated spectral projection. Note that $\norm{h - h(N)} \leq \frac 1N$.

\begin{lemma}
	\label{lem:BHL-pre}
	Given $\gamma > 0$ and  $N \in \N$ with $\gamma < \frac 1 {16 (N+1)}$, there exists $\delta > 0$ such that the following holds: Whenever $h$ is a positive contraction with finite spectrum in a unital $\K_1$-simple real rank zero $\cstar$-algebra\footnote{See \cite{BEEK} for the definition of $\K_1$-simple real rank zero $\cstar$-algebras. Let us note that if $A$ is real rank zero, i.e. the set of self-adjoint invertible elements of $A$ is dense in the set of invertible elements of $A$, and  $\K_1(A)$ is trivial, then $A$ is $\K_1$-simple if the map $V(A) \to \K_0(A)$ is injective.} $A$ and $u \in A$ is a unitary such that 
	$$
		\norm{uhu^* - h} < \delta,
	$$
	then there exists a unitary path $(u_t)_{t \in [0,1]}$ in $A$ such that $u_0 = 1_A, \   u_1 h(N) u_1^* =  uh(N)u^*$ and 
	$$
				\norm{u_th(N)u_t^* - h(N)} < \frac 6 N + 64 (N+1)\gamma + 2 \delta   \qquad (t \in [0,1]).
	$$
\end{lemma}

\begin{cor}
	\label{cor:BHL}
		For every $\e > 0$, there exists $\delta > 0$ such that  the following holds: Whenever $h$ is a positive contraction in a $\T{UHF}$-algebra $U$ and $u \in U$ is a unitary such that
		$$
			\norm{uhu^* - h} < \delta,
		$$
		then there exists a unitary path $(u_t)_{t \in [0,1]}$ in $U$, such that  $u_0 = 1_U,  \ u_1 = u$ and 
	$$
				\norm{u_thu_t^* - h} < \e \qquad ( t \in [0,1]).
	$$
\end{cor}

\begin{proof}
	Let $0<\e<1$. Choose $N \in \N$ such that $\frac 6 N \leq \frac \e {30}$ and let $\gamma \coloneqq \frac \e {30 \cdot 64(N+1)}$. Then, there exists $0 < \delta < \frac \e {60}$, such that the conclusion of Lemma \ref{lem:BHL-pre} holds for $\gamma$ and $N$. Note that $\T{UHF}$-algebras are $\K_1$-simple real rank zero $\cstar$-algebras.  We show that $\frac \delta 3$ does the job.  To this end let $h$ be a positive contraction and let $u$ be a unitary in some $\T{UHF}$-algebra $U$ with 
	$$
		\norm{uhu^* - h} < \frac \delta 3.
	$$
	Since $U$ has real rank zero, we can find  a positive contraction $k$  with finite spectrum such that $\norm{h-k} < \frac \delta 3$. Then 
	$$
		 \norm{uku^* -k} < \delta.
	$$
	In particular, $k$ is a positive contraction with finite spectrum to which Lemma \ref{lem:BHL-pre} applies, i.e. there exists a unitary path $(z_t)_{t \in [0,1]}$ such that $z_0 = 1_U, \  z_1 k(N) z_1^* = u k(N)u^*$ and 
		$$
				\norm{z_t k(N) z_t^* - k(N)} <\frac 6 N + 64(N+1)\gamma + 2 \delta < \frac \e {10}  \qquad (t \in [0,1]).
		$$
	 Let $w := z_1^*u$. Then, the unitary path $(v_t)_{t \in [0,1]}$, where  $v_t \coloneqq   z_t^*u$,   connects $u$ with $w$ and one easily checks that 
	$$
		\norm{v_t h v_t^* -h } < \e \qquad (t \in [0,1]).
	$$
	Using that $z_1k(N)z_1^* = uk(N)u^*$, it follows that $w$ actually commutes with $k(N)$ and since $k(N)$ has finite spectrum, we can connect $w$ to $1_U$ via a unitary path $(w_t)_{t \in [0,1]}$ such that $w_t$ commutes with $k(N)$ for every $t \in [0,1]$. Then,  
	$$
		\norm{w_thw_t^* -h} < \e \qquad (t \in [0,1]).
	$$
	The desired path $(u_t)_{t \in [0,1]}$ is now given by 
	$$
		u_t = \begin{cases}
			v_{2t} & \text{ if }  t \in [0,1/2], \\
			w_{2t-1} & \text{ if }  t \in [1/2,1].
		\end{cases}
	$$
\end{proof}

As an easy consequence we get the following:

\begin{lemma}
	\label{lem:BHL-central}
	For every $\e > 0$ and $\cF \subseteq C([0,1])$ finite,  there exists $\delta > 0$ such that the following holds: Whenever $\varphi \colon C([0,1]) \to U$ is a unital $\Star$-homomorphism, where $U$ is a $\T{UHF}$-algebra, and $u \in U$ is a unitary such that 
	$$
		\norm{u \varphi(\iota) u^* - \varphi(\iota)} < \delta,
	$$
	then there exists a unitary path $(u_t)_{t \in [0,1]}$ in $U$  such that $u_0 = 1_U, \ u_1 = u$ and 
	$$
		\norm{u_t \varphi(f) u_t^* - \varphi(f)} < \e  \qquad (f \in \cF, \ t \in [0,1]).
	$$
	By $\iota$ we denote the canonical generator of $C_0(0,1]$.
\end{lemma}

	The following is a Basic Homotopy Lemma for maps $Z_{\mf p, \mf q} \to M_{\mf p} \otimes M_{\mf q}$.

\begin{lemma}
	\label{lem:BHL-dimension-drop}
	For every $\e > 0$ and   $\cF \subseteq Z_{\mf p, \mf q}$ finite, where $\mf p$ and  $\mf q$ are supernatural numbers of infinite type, there exists $\delta > 0$ and $\cG \subseteq Z_{\mf p, \mf q}$ finite such that the following holds:  Whenever $\varphi \colon Z_{\mf p, \mf q} \to U$ is a unital $\Star$-homomorphism, where $U$ is a $\T{UHF}$-algebra of infinite type that absorbs $M_\mf p$ and $M_\mf q$ tensorially, and $u \in U$ is a unitary such that 
	$$
		\norm{u\varphi(g)u^* - \varphi(g)} < \delta \qquad (g \in \cG),
	$$
	then, there exists a unitary path $(u_t)_{t \in [0,1]}$ in $U$ with $u_0 = 1_U, \ u_1 = u$ and 
	$$
		\norm{u_t\varphi(f)u_t^* - \varphi(f)} < \e \qquad (f \in \cF, \ t \in [0,1]).
	$$
\end{lemma}

\begin{proof}
	Let $\e > 0$ and $\cF \subseteq Z_{\mf p, \mf q}$ be finite. We may assume that 
	$$
		\cF = \acute \cF \cup \grave \cF , 
	$$ where $\acute \cF = \{g_k \otimes y_k\}_{k=1}^N \subseteq C_0(0,1] \otimes M_\mf q$  and $\grave \cF = \{f_k \otimes x_k \}_{k=1}^N \subseteq C_0[0,1) \otimes M_\mf p$. The two cones are identified as subalgebras of $Z_{\mf p, \mf q}$ as explained in \ref{def:dimension-drop}.  With (\ref{eq:W-equation}), we  choose $t_0 \in [0,1)$ such that
	\begin{eqnarray}
		\label{eq:twist-estimate}
		\begin{aligned}
		\acute  W_{t_0} (g_k \otimes y_k \otimes 1_\mf q) (\acute W_{t_0})^* &  \approx_{\frac \e{10}} g_k \otimes 1_\mf q \otimes y_k \qquad (k = 1,2,\cdots , N),  \\ 
	\grave   W_{t_0} (f_k \otimes x_k \otimes 1_\mf p) (\grave W_{t_0})^*  & \approx_{\frac \e{10}} f_k \otimes 1_\mf p \otimes x_k \qquad (k = 1,2,\cdots , N) .
	\end{aligned}
	\end{eqnarray}
	The unitaries $\acute W_{t_0}$ and $\grave W_{t_0}$ are defined in \eqref{eq:W-def}. As observed in (\ref{eq:W-def}) the scalar part of $\acute W_{t_0}$ and $\grave W_{t_0}$ is $1$, so that we may approximate 
	\begin{eqnarray}
		\label{eq:W-approx-BHL}
			\begin{aligned}
		\grave  W_{t_0} -  1 \approx_{\frac \e {100}} \sum_{i=1}^K h_i \otimes a_i \otimes b_i , \\  
		\acute  W_{t_0} -1 \approx_{\frac \e {100}} \sum_{i=1}^K k_i \otimes c_i \otimes d_i, 
			\end{aligned}
	\end{eqnarray}
	where $h_i \in C_0[0,1), \ a_i,b_i  \in M_\mf p$ and $k_i \in C_0(0,1], \ c_i,d_i \in M_\mf q$. Furthermore, we clearly may assume that the elements $a_i,b_i,c_i,d_i$ are normalized. We then define 
	$$
		\cG \coloneqq \{h_i \otimes a_i\}_{i=1}^K \cup \{k_i \otimes c_i\}_{i=1}^K \cup \{\iota \} \subseteq Z_{\mf p, \mf q},
	$$
	 where $\iota$ is the canonical generator of $C_0(0,1]$. Let $0<\delta < \frac \e {100K}$ be as in Lemma \ref{lem:BHL-central},  with $\e$ replaced by $\frac \e {100}$ and $\cF$ replaced by $\{f_k\}_{k=1}^N \cup \{g_k\}_{k=1}^N$. Let us check that this choice of $\cG$ and $\delta$ works.
	\par To this end assume that $\varphi \colon Z_{\mf p, \mf q} \to U$ is a unital $\Star$-homomorphism, where $U \cong U \otimes M_\mf p \otimes M_\mf q$, and $u \in U$ is a unitary such that 
	\begin{equation}
		\label{eq:u-commutes-G}
		\norm{u\varphi(g)u^* - \varphi(g)} < \delta \qquad ( g \in \cG).
	\end{equation}
	Since $\iota \in \cG$, it follows from the choice of $\delta$, that there exists a unitary path $(u_t)_{t \in [0,1]}$ in $U$, such that  $u_0 = 1_U, \ u_1 = u$ and 
	\begin{equation}
		\label{eq:central-estimate}
		\norm{u_t\varphi(f)u_t^* - \varphi(f)} < \frac \e {100} \qquad (f \in \{f_k\}_{k=1}^N \cup \{g_k\}_{k=1}^N ).
	\end{equation}		
	Let us define  a unitary path 
	$$
	 Z_t \coloneqq (\acute U_{t_0}^\varphi)^* (\grave U_{t_0}^\varphi)^*  (u_t \otimes 1_{\mf p \mf q}) \acute U_{t_0}^{\varphi} \grave U_{t_0}^{\varphi}, \qquad (t \in [0,1]),
	$$
	with $\grave U_{t_0}^\varphi$ and $\acute U_{t_0}^\varphi$ as in (\ref{eq:U-def-pq}). We show that $Z_t$ commutes sufficiently well with $\varphi( \cF)$. Let us check this for $\acute \cF$. The computation for $\grave \cF$ is analogous. For $t \in [0,1]$  and $k = 1,2,\cdots, N$ we get:
	\begin{align*}
	 & 	\ad(Z_t) \circ (\acute \varphi (g_k \otimes y_k) \otimes 1_{\mf p \mf q }) \\ 
	 	\overset{\phantom{(\ref{eq:twist}),(\ref{eq:twist-estimate}) \quad }} {=_{\phantom{\frac \e {100}}}}  &  (\acute U_{t_0}^\varphi)^* (\grave U_{t_0}^\varphi)^*  (u_t \otimes 1_{\mf p \mf q}) \acute U_{t_0}^{\varphi} \grave U_{t_0}^{\varphi} (\acute \varphi (g_k \otimes y_k) \otimes 1_{\mf p \mf q})  (\grave U_{t_0}^{\varphi})^* (\acute U_{t_0}^{\varphi})^* (u_t^* \otimes 1_{\mf p \mf q}) \grave U_{t_0}^{\varphi} \acute U_{t_0}^{\varphi} \\
	 	 \overset{(\ref{eq:twist}),(\ref{eq:twist-estimate})\quad }{ \approx_{\frac \e {100}}} 	&  (\acute U_{t_0}^{\varphi})^* (\grave U_{t_0}^{\varphi})^*  (u_t\otimes 1_{\mf p \mf q}) (\varphi(g_k) \otimes 1_\mf p \otimes y_k) (u_t^* \otimes 1_{\mf p \mf q}) \grave U_{t_0}^{\varphi} \acute U_{t_0}^{\varphi} \\
	 	 \overset{\phantom{(2)}(\ref{eq:central-estimate}) \phantom{(2.18)}}{ \approx_{\frac \e {100}} } &  (\acute U_{t_0}^{\varphi})^* (\grave U_{t_0}^{\varphi})^*  (\varphi (g_k) \otimes 1_\mf p \otimes y_k)   \grave U_{t_0}^{\varphi} \acute U_{t_0}^{\varphi} \\
	 	 \overset{(\ref{eq:twist}),(\ref{eq:twist-estimate}) \quad }{ \approx_{\frac \e {100}}} & \acute \varphi(g_k \otimes y_k) \otimes 1_{\mf p \mf q}.
	 \end{align*}
	It follows that  
	\begin{equation}
		\label{eq:almost-commuting-unitary-path}
		\norm{Z_t \varphi(f) Z_t^* - \varphi(f)} < \frac \e {30} \qquad (f \in \cF , \ t \in [0,1]).
	\end{equation}		
	Next, we show that $Z_1$ is close to $u \otimes 1_{\mf p \mf q}$. To this end, observe that 
	\begin{eqnarray}
		\label{eq:almost-u}	
			\begin{aligned}
		\norm{(\grave U_{t_0}^\varphi)^* (u \otimes 1_{\mf p \mf q}) \grave U_{t_0}^\varphi - u \otimes 1_{\mf p \mf q} } < \frac \e {30}  , \\  \norm{(\acute  U_{t_0}^\varphi)^* (u \otimes 1_{\mf p \mf q}) \acute U_{t_0}^\varphi - u  \otimes 1_{\mf p \mf q}} < \frac \e {30}.
			\end{aligned}
	\end{eqnarray}
		Let us prove that (\ref{eq:almost-u}) holds for $\grave U_{t_0}^\varphi$. The computation for $\acute U_{t_0}^\varphi$ is similar. Since the image of the scalar part is central it suffices to show that (recall the definition of $\grave U_{t_0}^\varphi$ from (\ref{eq:U-def-pq})): 
		$$
			(u \otimes 1_{\mf p \mf q}) \cdot (\grave \varphi \otimes \imath_{(\mf p, \mf q)}^{[1]})^+(\grave W_{t_0} -1) \approx_{\frac \e {30}} (\grave \varphi \otimes\imath_{(\mf p, \mf q)}^{[1]})^+(\grave W_{t_0} -1) \cdot (u \otimes 1_{\mf p \mf q}).
		$$
		The following calculation shows that this is indeed true:
		\begin{align*}
			(u \otimes 1_{\mf p \mf q}) \cdot (\grave \varphi \otimes \imath_{(\mf p, \mf q)}^{[1]})^+(\grave W_{t_0} -1) &  
			 \overset{(\ref{eq:W-approx-BHL}) \quad}{ \approx_{\frac \e {100}}} ( u \otimes 1_{\mf p \mf q})  \cdot  (\grave \varphi \otimes \imath_{(\mf p, \mf q)}^{[1]})^+ \left ( \sum_{i=1}^K    h_i \otimes a_i \otimes b_i \right )   \\
			 & \overset{\phantom{(2.11) \quad }}{=_{\phantom{\frac \e {100}}}}  \sum_{i=1}^K  u \grave \varphi(h_i \otimes a_i)   \otimes b_i \otimes 1_{\mf q}      \\
			&  \overset{(\ref{eq:u-commutes-G}) \quad }{  \approx_{K \delta }} \sum_{i=1}^K    \grave \varphi(h_i \otimes a_i) u   \otimes b_i \otimes 1_{\mf q}    \\
			& \overset{\phantom{(2.11) \quad }}{=_{\phantom{\frac \e {100}}}} (\grave \varphi \otimes \imath_{(\mf p, \mf q)}^{[1]})^+ \left ( \sum_{i=1}^K h_i \otimes a_i \otimes b_i \right ) \cdot (u \otimes 1_{\mf p \mf q}) \\
		& 	 \overset{(\ref{eq:W-approx-BHL}) \quad}{  \approx_{\frac \e {100}}}    (\grave \varphi \otimes \imath_{(\mf p, \mf q)}^{[1]})^+(\grave W_{t_0} -1) \cdot ( u \otimes 1_{\mf p \mf q}).
		\end{align*}
		Since $K\delta < \frac \e {100}$, the claim follows. We are now ready to construct the desired unitary path in $U$. By \cite[Remark 2.7]{TW}, there exists a unital $\Star$-homomorphism $\theta \colon U \otimes M_\mf p \otimes M_\mf q \to U$ such that 
	$$
		\theta(x \otimes 1_{\mf p \mf q}) \approx_{\frac \e {10}} x \qquad ( x \in \varphi(\cF) \cup \{u \}).
	$$
	Let $(W_t)_{t \in [0,1]}$ be given by  $W_t \coloneqq \theta(Z_t)$. Then,
	\begin{align*}
		W_t \varphi(f) W_t^* & \approx_{\frac \e {10}} \theta(Z_t (\varphi(f) \otimes 1_{\mf p \mf q}) Z_t^*)  \\
	& \approx_{\frac \e {30} } \theta(\varphi(f) \otimes 1_{\mf p \mf q})  \\
	& \approx_{\frac \e {10}} \varphi(f),
	\end{align*}
	for all $f \in \cF$ and  $t \in [0,1]$. Next, by (\ref{eq:almost-u}), it follows that   $\norm{Z_1 - u \otimes 1_{\mf p \mf q}} < \frac \e {15}$, so that $W_1 = \theta(Z_1) \approx_{\frac \e {15}} \theta(u \otimes 1_{\mf p \mf q}) \approx_{\frac \e {10}} u$.  By a standard functional calculus argument, we may extend $(W_t)_{t \in [0,1]}$ to a unitary path $(W_t)_{t \in [0,2]}$ such that $W_2 = u$ and $\norm{W_t - W_1} < \frac \e 6$, for $t \in [1,2]$. Then 
	$$
		\norm{W_t \varphi(f) W_t^* - \varphi(f)} < \e \qquad (t \in [0,2], \  f \in \cF ).
	$$
	This finishes the proof.
\end{proof}

\begin{theorem}
	\label{th:uniqueness}
	Let $\mf p$ and $\mf q$ be supernatural numbers of infinite type. Assume that $\varphi, \psi  \colon Z_{\mf p, \mf q} \to U$ are standard $\Star$-homomorphisms, where $U$ is a $\T{UHF}$-algebra of infinite type that absorbs $M_\mf p$ and $M_\mf q$ tensorially. Then  $\varphi \sim_{\T{asu}} \psi$.
\end{theorem}

\begin{proof}
	By Lemma \ref{lem:aue-dimension-drop} we know that $\varphi \approx_{\T u} \psi$. Furthermore we have established a Basic Homotopy Lemma (Lemma \ref{lem:BHL-dimension-drop}) for $\varphi$ and $\psi$. By a standard argument, these two ingredients combined show that $\varphi  \sim_{\T{asu}} \psi$. The proof is exactly the same as in \cite[Theorem 2.2]{DW}. See also the proof of Theorem \ref{th:aue-endo} for a similar argument.
\end{proof}

\begin{theorem}
	\label{th:aue-endo}
	Let $\mf p$ and $\mf q$ be supernatural numbers of infinite type. Assume that $\varphi, \psi \colon Z_{\mf p, \mf q} \to Z_{\mf p, \mf q}$ are standard $\Star$-homomorphisms. Then $\varphi \approx_{\T u} \psi$.
\end{theorem}

\begin{proof}
	Let $\cF \subseteq Z_{\mf p, \mf q}$ be finite and $\e > 0$. Let $\cG \subseteq Z_{\mf p, \mf q}$ be finite and $0 < \delta < \e $, such that the conclusion of Lemma \ref{lem:BHL-dimension-drop} holds for $\cF$ and $\e$ replaced by $\frac \e {10}$.  Let $N \in \N$ such that 
	$$
		\abs{s-t} \leq  \frac 1 N \quad \Rightarrow \quad \norm{\gamma_s(f) - \gamma_t(f)} < \frac \delta {10} \qquad (f \in \cF \cup \cG, \ \gamma \in \{\varphi,\psi\} ).
	$$
	 Since $\varphi_s, \psi_s \colon Z_{\mf p, \mf q} \to M_\mf p \otimes M_\mf q$ are standard for each $s \in [0,1]$, we know from Lemma \ref{lem:aue-dimension-drop} that 
	$$
		\varphi_{\frac i N } \approx_{\T u} \psi_{\frac i N} \qquad (i = 0,1,2,\cdots,N-1,N).
	$$
	It follows that there exist unitaries $u_0,u_1,\cdots, u_{N-1},u_N \in M_\mf p \otimes M_\mf q$, such that 
	$$
		\norm{u_i \varphi_{\frac i N}(g) u_i^*  - \psi_{\frac i N}(g)} < \frac \delta 5 \qquad (g \in \cG).
	$$
	We may assume that $u_0 \in M_\mf p \otimes 1_\mf q$ and $u_N \in 1_\mf p \otimes M_\mf q$. Now, observe that for $i = 0,1,2,\cdots, N-1$, we have 
	$$
		\norm{u_{i+1}^*u_i \varphi_{\frac i N}(g) u_i^* u_{i+1} - \varphi_{\frac i N}(g)} < \delta \qquad (g \in \cG).
	$$
	By the choice of $\cG$ and $\delta$, it follows that there exist unitary paths $(u_{i,t})_{t \in [0,1]}$ for $i=0,1,\cdots,N-1$, with $u_{i,0} = 1_{\mf p \mf q}, \ u_{i,1} = u_{i+1}^*u_i$ and such that:
	$$
		\norm{u_{i,t} \varphi_{\frac i N}(f) u_{i,t}^* - \varphi_{\frac i N}(f)} < \frac \e {10} \qquad (t \in [0,1], \ f \in \cF).
	$$
	This defines a unitary $u$ in $Z_{\mf p, \mf q}$ by 
	$$
		u_t \coloneqq u_i u_{i,Nt-i}^* \quad \text{ for } \quad  t \in \left [ \frac i N , \frac {i+1} N \right ].
	$$
	For $f \in \cF $ and $t \in [i/N, (i+1)/N]$, we get 
	\begin{align*}
		u_t \varphi_t(f) u_t^* & \approx_{\frac \delta {10}}  	u_t \varphi_{\frac iN}(f) u_t^*   = u_ i u_{i,Nt-i}^* \varphi_{\frac iN}(f) u_{i,Nt-i} u_i^* \\
		& \approx_{\frac \e {10} } u_i \varphi_{\frac i N}(f) u_i^*  \approx_{\frac \delta 5 } \psi_{\frac i N}(f) \\
		& \approx_{\frac \delta {10}} \psi_t (f).
	\end{align*}
	As this holds on every interval $[i/N, (i+1)/N]$, it follows that 
	$$
		\norm{u\varphi(f)u^* - \psi(f)} < \e \qquad (f \in \cF).
	$$
	This proves that $\varphi \approx_{\T u} \psi$.
\end{proof}

\section{Existence}

In this section we show how to construct standard $\Star$-homomorphisms from generalized dimension drop algebras into certain UHF-algebras. Furthermore, using our asymptotic uniqueness result (Theorem \ref{th:uniqueness}), we prove the existence of so-called \textit{unitarily suspended} $\Star$-homomorphisms between generalized dimension drop algebras (Theorem \ref{th:existence-unitarily-suspended}). They will be the key ingredient in proving the main theorem of this paper (Theorem \ref{th:SSA}).
\ \\
\begin{convention} Let $x \in M_n$ and $y \in M_m$, where $n,m \in \N$.  Whenever we identify $M_n \otimes M_m$ with $M_{nm}$,  we do this via 
		$$
			x \otimes y = (x \cdot y_{ij})_{i,j=1}^m.
		$$
		For $p \in \N$, define the UHF-algebra $M_{p^\infty}$ to be the inductive limit of 
	$$
		M_{p^n} \to M_{p^{n+1}} : x \mapsto  1_p \otimes x. 
	$$
\end{convention}

\begin{lemma}
	\label{lem:lebesgue-contractions}
	 Let $A$ be a unital $\cstar$-algebra. Assume $0 < \alpha \leq \beta \leq 1$. If $h \in A$ is a Lebesgue contraction (see Definition \ref{def:Lebesgue-Contraction}), then the element
	$$
		h(\alpha,\beta) \coloneqq \alpha 1_A + (\beta-\alpha) h
	$$
	has spectrum $[\alpha,\beta]$ and 
	$$
		\tau(f(h(\alpha,\beta))) = \frac 1 {\beta-\alpha} \int_\alpha^\beta f(t) \ dt \qquad \big ( f \in C_0(0,1] , \ \tau \in T(A) \big ) .
	$$
\end{lemma}

\begin{proof}
	Using the Binomial Theorem, one easily checks that the the conclusion of the lemma holds for polynomials. By an approximation argument, the statement follows.
\end{proof}

\begin{remark}
	\label{rem:cone-universal}
	(\cite[{}1.2.3]{WICOV}) Let $p \geq 2$ be a natural number. For $i,j = 2,\cdots,p$ we consider the relations
	\[
		\norm{x_i} \leq 1, \ x_ix_j = 0 \ (i \neq j), \ x_i^*x_j = \delta_{i,j} \cdot x_2^*x_2. 
			\tag{$\mathcal R_p$}
	\]
	Then 
	$$
		C_0(0,1] \otimes M_p \cong \cstar(x_2,x_3,\cdots,x_p \mid \mathcal R_p).
	$$	
	In particular, if $x_2,x_3,\cdots,x_p$ are elements in some $\cstar$-algebra $A$ satisfying $(\mathcal R_p)$, then the assignment 
	$$
		\varphi (e_{i1}) = x_i^*x_i \qquad ( i=2,3,\cdots,p)
	$$
	defines a c.p.c. order zero map $\varphi \colon M_p \to A$. Indeed, the elements $ \{ \iota^\frac 1 2  \otimes e_{i1} \}_{i=2}^p$ satisfy the relations $(\mathcal R_p)$ and hence  $\overline \varphi : C_0(0,1] \otimes M_p \to A : \iota^\frac 12 \otimes e_{i1} \mapsto x_i$ defines a $\Star$-homomorphism. Note that 
	$$
		\varphi(e_{11}) = \varphi(e_{i1}^* e_{i1}) = x_i^*x_i = x_2^*x_2 \qquad (i = 2,3,\cdots,p). 
	$$
	Furthermore, for $i,j =2,3,\cdots,p$ we get 
	$$
		\varphi(e_{ij}) = \varphi(e_{i1}e_{j1}^*) = x_ix_j^*.
	$$
\end{remark}	

\begin{lemma}
	\label{lem:universal-elements}
	Let $p,q \geq 2$ be natural numbers. Then $M_{q^\infty}$ contains elements $x_2,x_3,\cdots,x_p$ satisfying $(\cR_p)$  such that $x_2^*x_2$ has $\frac 1 p$ Lebesgue spectral measure.
\end{lemma}

\begin{proof}
	Let us recursively define two sequences of non-negative integers $(\alpha_i)_{i=1}^\infty$ and $(\beta_i)_{i=0}^\infty$. Set $\beta_0 \coloneqq 1 $ and define $\alpha_i$ and $\beta_i$ for $i\geq 1$ by 
	\begin{equation}
		\label{eq:recursion}
		q \beta_{i-1} = p\alpha_i + \beta_i,
	\end{equation}
	where $\beta_i \in \{0,1,2,\cdots,p-1\}$. A short computation using (\ref{eq:recursion})  shows that the coefficients $(\alpha_i)_{i=1}^\infty$ are chosen such that 
	\begin{equation}
		\label{eq:sum}
		\sum_{i=1}^\infty \frac{p \alpha_i}{q^i} = 1.
	\end{equation}
	For further reference, we denote the partial sums by $S_N$. We define recursively another sequence $(\delta_i)_{i=0}^\infty$  by $\delta_0 = 0$ and for $i \geq 1$ we let 
	\begin{equation}
		\label{eq:recursion-2}
		\delta_i := q ( \delta_{i-1} + p \alpha_i ).
	\end{equation}
	Using (\ref{eq:recursion-2}) one easily checks that 
	\begin{equation}
		\label{eq:recursion-3}
			q^{i+1} = q\beta_i + \delta_i \qquad (i \geq 0).
	\end{equation}
		In particular (\ref{eq:recursion-3}) implies that 
		\begin{equation}
			\label{eq:positive-remainder}
			\delta_{i-1} + p\alpha_i +\beta_i  =  q^i \qquad( i \geq 1).
		\end{equation}
		All this is made in such a way that we can subdivide the matrix algebra $M_{q^i}$ in three orthogonal parts, as indicated in Figure \ref{fig:contractions}.
	\begin{figure}[h]
			\centering
			\tikzset{every picture/.style={line width=0.25pt}} 

\tikzset{every picture/.style={line width=0.25pt}} 

\begin{tikzpicture}[x=0.3pt,y=0.3pt,yscale=-1,xscale=1]

\draw    (240,180) -- (390.68,180) ;

\draw    (390,30) -- (390,180) ;

\draw   (390.8,180) -- (540.68,180) -- (540.68,330) -- (390.8,330) -- cycle ;
\draw    (540.68,330) -- (660,330) ;

\draw    (540.68,450) -- (540.68,330) ;

\draw   (210.2,30.4) .. controls (205.53,30.36) and (203.18,32.67) .. (203.14,37.34) -- (202.68,95.04) .. controls (202.63,101.71) and (200.27,105.02) .. (195.6,104.99) .. controls (200.27,105.02) and (202.57,108.37) .. (202.52,115.04)(202.54,112.04) -- (202.06,172.74) .. controls (202.02,177.41) and (204.33,179.76) .. (209,179.8) ;
\draw   (210.2,180.2) .. controls (205.53,180.19) and (203.19,182.51) .. (203.18,187.18) -- (203.03,244.58) .. controls (203.02,251.25) and (200.68,254.57) .. (196.01,254.56) .. controls (200.68,254.57) and (203,257.91) .. (202.98,264.58)(202.99,261.58) -- (202.83,321.98) .. controls (202.82,326.65) and (205.14,328.99) .. (209.81,329) ;
\draw   (209.8,329.8) .. controls (205.13,329.81) and (202.81,332.15) .. (202.82,336.82) -- (202.97,379.62) .. controls (202.99,386.29) and (200.67,389.63) .. (196,389.65) .. controls (200.67,389.63) and (203.01,392.95) .. (203.03,399.62)(203.02,396.62) -- (203.17,442.42) .. controls (203.19,447.09) and (205.53,449.41) .. (210.2,449.4) ;
\draw  [draw opacity=0] (390.68,180) -- (540.68,180) -- (540.68,330) -- (390.68,330) -- cycle ; \draw   (390.68,180) -- (390.68,330)(428.68,180) -- (428.68,330)(466.68,180) -- (466.68,330)(504.68,180) -- (504.68,330) ; \draw   (390.68,180) -- (540.68,180)(390.68,218) -- (540.68,218)(390.68,256) -- (540.68,256)(390.68,294) -- (540.68,294) ; \draw    ;
\draw   (378.78,180) .. controls (374.11,179.99) and (371.77,182.31) .. (371.76,186.98) -- (371.56,244.98) .. controls (371.54,251.65) and (369.2,254.97) .. (364.53,254.95) .. controls (369.2,254.97) and (371.52,258.31) .. (371.5,264.98)(371.51,261.98) -- (371.3,322.98) .. controls (371.29,327.65) and (373.61,329.99) .. (378.28,330) ;
\draw   (390.63,336.35) .. controls (390.63,341.02) and (392.96,343.35) .. (397.63,343.35) -- (400.27,343.35) .. controls (406.94,343.35) and (410.27,345.68) .. (410.27,350.35) .. controls (410.27,345.68) and (413.6,343.35) .. (420.27,343.36)(417.27,343.36) -- (422.46,343.36) .. controls (427.13,343.36) and (429.46,341.03) .. (429.46,336.36) ;
\draw  [line width=0.5]  (240,30) -- (660,30) -- (660,450) -- (240,450) -- cycle ;
\draw   (119.2,31.4) .. controls (114.53,31.37) and (112.18,33.68) .. (112.15,38.35) -- (110.92,230.16) .. controls (110.87,236.82) and (108.52,240.14) .. (103.85,240.11) .. controls (108.52,240.14) and (110.83,243.48) .. (110.79,250.15)(110.81,247.15) -- (109.55,441.95) .. controls (109.52,446.62) and (111.83,448.97) .. (116.5,449) ;

\draw (315,105.32) node [scale=0.75]   {$0$};
\draw (600,390) node [scale=0.75]   {$0$};
\draw (166,106) node [scale=0.75]   {$\delta _{i-1}$};
\draw (167,256) node [scale=0.75]   {$p\alpha _{i}$};
\draw (169,390) node [scale=0.75]   {$\beta _{i}$};
\draw (345,256) node [scale=0.75]   {$p$};
\draw (408.36,370.4) node  [scale=0.75]  {$M_{\alpha _{i}}( M_{q^{\infty }})$};
\draw (77,240) node  [scale=0.75]  {$q^{i}$};
\draw (410,200) node [scale=0.75]   {$\tilde{h}_{i}$};
\end{tikzpicture}
			\caption{Subdivision of $M_{q^i}$ and the element  $\bar h_i$}
			\label{fig:contractions}
	\end{figure}
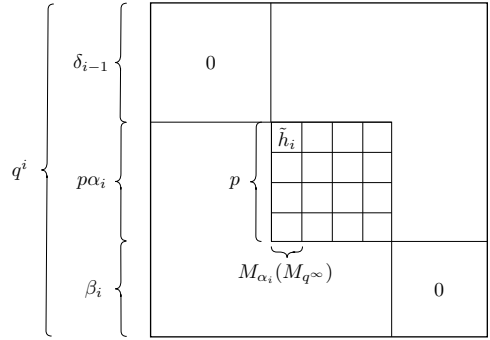
	 Let $h \in M_{q^\infty}$ be a Lebesgue contraction, for example 
		$$
		h \coloneqq  \sum_{i=1}^\infty \left (  \sum_{k=1}^{q-1} \frac k {q^i} 1_q^{\otimes (i-1)} \otimes  e_{kk}^{(q)} \otimes 1_q^{\otimes \infty} \right ) .
	$$	 
	  For $i \in \N$ and $r \in \{1,2,\cdots,\alpha_i\}$ we set 
		$$
			\tilde h_{i,r} \coloneqq h \left  (1-S_{i-1}-\frac {rp}{q^i},1-S_{i-1} - \frac{(r-1)p}{q^i}  \right  )  \in M_{q^\infty} 
		$$
		and 
	$$
		\tilde h_i :=  \diag ( \tilde h_{i,1}, \tilde h_{i,2}, \cdots, \tilde h_{i,\alpha_i} )  \in M_{\alpha_i}(M_{q^\infty}).
	$$
	By Lemma \ref{lem:lebesgue-contractions}, we know 
	$$
		\tau_{M_{q^\infty}}  ( f(\tilde h_{i,r})  ) =  \frac {q^i} p \int_{1-S_{i-1} - \frac{rp}{q^i}}^{1-S_{i-1} - \frac{(r-1)p}{q^i}}f(t) \ dt \qquad (f \in C_0(0,1], \ 1 \leq r \leq \alpha_i )
	$$
	and hence 
	\begin{equation}
		\label{eq:integral}
	(\tau_{M_{q^\infty}} \circ \T{Tr}_{\alpha_i})(f(\tilde h_i)) = \frac {q^i} p \int_{1-S_i}^{1-S_{i-1}} f(t) \ dt \qquad (f \in C_0(0,1]),
	\end{equation}
	where $\T{Tr}_{\alpha_i}$ denotes the sum over the diagonal in $M_{\alpha_i}(M_{q^\infty})$. Next, put $\tilde h_i$ into the upper left corner of $M_p( M_{\alpha_i} ( M_{q^\infty}))$,  which we denote by 
	$$
		e_{11}^{(p)} \otimes \tilde h_i  
	$$
	and set 
	$$
			\bar  h_i \coloneqq 0_{\delta_{i-1}} \oplus (e_{11}^{(p)} \otimes \tilde h_i) \oplus 0_{\beta_i} \in M_{q^i}(M_{q^\infty}) \cong M_{q^i} \otimes M_{q^\infty}.
	$$
	The element $\bar h_i$ is as depicted in Figure  \ref{fig:contractions}. Note that by (\ref{eq:positive-remainder}), the elements $\bar h_i$ live in $M_{q^i}(M_{q^\infty})$. Next, we consider the elements $\bar h_i$ in
	$$
		M_{q^\infty} \otimes M_{q^\infty} = \varinjlim \ ( M_{q^i} \otimes M_{q^\infty}, 1_q \otimes \id_{M_{q^i}} \otimes \id_{M_{q^\infty}}).
	$$ 
	By construction, the $\bar h_i$ are mutually orthogonal positive contractions which are summable  in  $M_{q^\infty} \otimes M_{q^\infty}$. We therefore may define 
	$$
		  \bar h \coloneqq \sum_{i=1}^\infty \bar  h_i \in   M_{q^\infty} \otimes M_{q^\infty}.
	$$
	 We can now check that $\bar h$ has $\frac 1 p$ Lebesgue spectral measure. Denote $T_N \coloneqq \sum_{i=1}^N \bar h_i$. For $f \in C_0(0,1]$ and $N \in \N$ we get:
	\begin{align*}
			\tau_{M_{q^\infty} \otimes M_{q^\infty}}(f(T_N))   & = \sum_{i=1}^N (\tau_{M_{q^i}} \otimes \tau_{M_{q^\infty}})(f(\bar  h_i)) =  \sum_{i=1}^N \frac 1 {q^i} (\tau_{M_{q^\infty}} \circ \T{Tr}_{\alpha_i}) ( f(\tilde h_i)) \\
			&   \overset{(\ref{eq:integral})}  =  \sum_{i=1}^N \frac 1 {q^i}  \frac {q^i} p \int_{1-S_i}^{1-S_{i-1}} f(t) \ dt  =    \frac 1 p   \int_{1-S_N}^1 f(t) \ dt .
	\end{align*}
	By (\ref{eq:sum}) the result follows, if we let $N$ tend to infinity. 
	\par In order to finish the proof, instead of $\bar  h_i$, we look, for $j = 2,3, \cdots , p$, at the elements
	$$
		\bar x_{j,i}  \coloneqq 0_{\delta_{i-1}} \oplus (e_{j1}^{(p)} \otimes (\tilde  h_i)^{\frac 1 2}) \oplus 0_{\beta_i} \in M_{q^i}(M_{q^\infty}) \cong M_{q^i} \otimes M_{q^\infty} 
	$$
	and define 
	$$
		x_j \coloneqq \sum_{i=1}^\infty \bar x_{j,i} \qquad (j =2,3,\cdots,p).
	$$
	The elements $\{x_j\}_{j=2}^p$ then satisfy $(\cR_p)$ and $x_2^*x_2 = \bar h$, which has $\frac 1 p$ Lebesgue spectral measure.
\end{proof}

The reader is invited to run this construction for $p = 2$ and $q =3$, in which case $\alpha_i = \beta_i = 1$, for all $i \in \N$.

\begin{lemma}
	\label{lem:standard-cpc-order-zero}
	Let $p,q \geq 2$ be natural numbers. Then, there exists a c.p.c. order zero map 
	$$
		\varphi \colon M_p \to M_{q^\infty} 
	$$
	such that $h_\varphi$ is a Lebesgue contraction.
\end{lemma}

\begin{proof}
	This is now an immediate consequence of Lemma  \ref{lem:universal-elements} and the universal property of the cone $C_0(0,1] \otimes M_p$, as described in Remark \ref{rem:cone-universal}.
\end{proof}

\begin{lemma}
	\label{lem:commuting-cones}
	Let $p,q \geq 2$ be natural numbers. Then, for $r \in \{p,q\}$, there exist c.p.c. order zero maps $\alpha \colon M_p \to M_{r^\infty}$ and $\beta \colon M_q \to M_{r^\infty}$ such that 
	$$
		\alpha(1_p) + \beta(1_q) = 1, \quad [\alpha(M_p),\beta(M_q)] = \{0\}
	$$
	and such that $h_\alpha,h_\beta$ are Lebesgue contractions.
\end{lemma}

\begin{proof}
	Let us do the case $r = q$. By Lemma \ref{lem:standard-cpc-order-zero}, there exists a c.p.c. order zero map $ \phi \colon M_p \to M_{q^\infty}$ such that $h_\phi$ is a Lebesgue contraction. We then define c.p.c. order zero maps  
	 \[
	\begin{array}{ll}
		\alpha \colon   M_p \to M_{q^\infty} \otimes M_q : x \mapsto  \phi(x) \otimes 1_q, \\
		\beta \colon  M_q \to M_{q^\infty} \otimes M_q : y \mapsto (1-h_\phi) \otimes y.
	\end{array}
	\]
	Since $h_\phi$ commutes with the image of $\phi$, we see that the images of $\alpha$ and $\beta$ commute. Furthermore, by definition $\alpha(1_p) + \beta(1_q) = 1$ and it is clear that $h_\alpha$ and $h_\beta$ are Lebesgue contractions.
\end{proof}

\begin{cor}
	\label{cor:existence-finite}
	Let $p,q \geq 2$ be coprime natural numbers. Then, for $r \in \{p,q\}$, there exists a standard $\Star$-homomorphism 
	$$
		\varphi \colon Z_{p,q} \to M_{r^\infty}.
	$$
\end{cor}

\begin{proof}
	 Let $r \in \{p,q\}$ and let $\alpha \colon M_p \to M_{r^\infty}$ and $\beta \colon M_q \to M_{r^\infty}$ be as in Lemma \ref{lem:commuting-cones}.  By \cite[Proposition 2.5]{RW} these maps induce a $\Star$-homomorphism $\varphi \colon Z_{p, q} \to M_{r^\infty}$, such that $\grave \varphi = (\overline \alpha)'$ and $\acute \varphi = \overline \beta$, where $(\cdot)'$ means that we have reversed the orientation of the unit interval. Since $\varphi$ is completely determined by $\acute \varphi$ and $\grave \varphi$, it follows that $\varphi$ is standard.
\end{proof}

\begin{convention}
	\label{def:dim-drop-limit}
	 Let $\mf p$ and  $\mf q$ be supernatural numbers. We then write $Z_{\mf p, \mf q}$ as an inductive limit as follows: Fix sequences $(P_n)_{n =1}^\infty , \ (Q_n)_{n =1}^\infty$ converging to $\mf p$ respectively $\mf q$ such that  $P_n$ divides $P_{n+1}$ and $Q_n$ divides  $Q_{n+1}$. This defines unital inclusions $M_{P_n} \subseteq M_{P_{n+1}}, \ M_{Q_n} \subseteq M_{Q_{n+1}}$ and $\gamma_{n,n+1} \colon Z_{P_n,Q_n} \to Z_{P_{n+1},Q_{n+1}}$ such that 
	$$
		Z_{\mf p, \mf q} = \varinjlim (Z_{P_n, Q_n},\gamma_{n,n+1}).
	$$
	Let us denote the limit maps by $\gamma_n \colon Z_{P_n,Q_n} \to Z_{\mf p, \mf q}$. Furthermore, we denote by 
	$$
		\kappa_n \colon Z_{\mf p, \mf q} \to Z_{P_n,Q_n}
	$$
	the induced conditional expectations.
\end{convention}

\begin{lemma}
	\label{lem:standard-supernatural}
	Le $\mf p$ and $\mf q$ be coprime supernatural numbers. Then, for $\mf r \in \{\mf p, \mf q\}$,  there exists a standard $\Star$-homomorphism
	$$
		\varphi \colon   Z_{\mf p, \mf q} \to M_{\mf r}.
	$$
\end{lemma}

\begin{proof}
	 Let $(\cF_n)_{n =1}^\infty$ be an increasing sequence of finite subsets of $Z_{\mf p, \mf q}$ with dense union, such that  $\cF_n \subseteq Z_{P_n,Q_n}$ and $\gamma_{n,n+1}(\cF_n) \subseteq \cF_{n+1}$. Furthermore, let $\e_n > 0$ with $\sum \e_n < \infty$. By Corollary \ref{cor:existence-finite}, for each $n \in \N$, there exists a standard $\Star$-homomorphism $\varphi_n \colon Z_{P_n,Q_n} \to M_\mf r$. Note that $\varphi_{n+1} \circ \gamma_{n,n+1}$ is also standard. Hence, by Lemma \ref{lem:aue-dimension-drop}, there exists a unitary $u_{n+1}\in M_\mf r$ such that 
	$$
		\norm{\varphi_n(x) - u_{n+1} \varphi_{n+1}( \gamma_{n,n+1}(x))u_{n+1}^*} < \e_n \qquad (x \in \cF_n).
	$$
	Define $\tilde \varphi_1 \coloneqq  \varphi_1$ and for $n \geq 2$ we let 
	$$
		\tilde \varphi_n \colon Z_{P_n,Q_n} \to M_\mf r : x \mapsto u_2u_3 \cdots u_n \varphi_n(x) u_n^* \cdots u_3^* u_2^*.
	$$
	One then checks that for each $x \in Z_{P_n,Q_n}$ the sequence $(\tilde \varphi_k(\gamma_{n,k}(x)))_{k =n}^\infty$ is Cauchy and we may define 
	$$
		\varphi (x) \coloneqq \lim_{k \to \infty} \tilde \varphi_k(\gamma_{n,k}(x)) \qquad (x \in Z_{P_n,Q_n}).
	$$
	This extends to a $\Star$-homomorphism $\varphi \colon Z_{\mf p, \mf q} \to M_\mf r$, which is easily seen to be standard.
\end{proof}

\begin{defn}
	\label{defn:unitarily-suspeded}
	Let $\mf p$ and $\mf q$ be supernatural numbers of infinite type.  We say that a  $\Star$-homomorphism $\mu \colon Z_{\mf p, \mf q} \to Z_{\mf p, \mf q}$ is \textit{unitarily suspended}, if 
		\[
			\mu_t = 
			\left \{
				\begin{array}{ll}
				 \T{ad}(u_t) \circ (\alpha_0 \otimes 1_\mf q)  & \text{ if } t \in [0,1) ,  \\
		 1_\mf p \otimes \alpha_1 & \text{ if } t = 1,
				\end{array}
				\right .
		\]
	where $\alpha_0 \colon Z_{\mf p, \mf q} \to M_\mf p$ and $\alpha_1 \colon Z_{\mf p, \mf q} \to M_\mf q$ are unital $\Star$-homomorphisms and $(u_t)_{t \in [0,1)}$ is a unitary path in $M_\mf p \otimes M_\mf q$ starting at the identity.
\end{defn}

\begin{theorem}
	\label{th:existence-unitarily-suspended}
	Let $\mf p$ and $\mf q$ be coprime supernatural numbers of infinite type. Then, there exists a unitarily suspended  standard $\Star$-homomorphism $\mu \colon Z_{\mf p, \mf q} \to Z_{\mf p, \mf q}$.
\end{theorem}

\begin{proof}
	By Lemma \ref{lem:standard-supernatural}, there exist standard $\Star$-homomorphisms $\alpha_0 \colon Z_{\mf p, \mf q} \to M_\mf p$ and $\alpha_1 \colon Z_{\mf p, \mf q} \to M_\mf q$. By Theorem \ref{th:uniqueness}, we have $\alpha_0 \otimes 1_\mf q \sim_{\T{asu}} 1_\mf p \otimes \alpha_1$, witnessed  by a unitary path $(u_t)_{t \in [0,1)}$ starting at the identity. Define a $\Star$-homomorphism $\mu \colon Z_{\mf p, \mf q} \to Z_{\mf p, \mf q}$ by $\mu_t \coloneqq \T{ad}(u_t) \circ (\alpha_0 \otimes 1_\mf p)$ for $t \in [0,1)$ and $\mu_1 \coloneqq 1_\mf p \otimes \alpha_1$. Then $\mu$ is standard and unitarily suspended in the sense of Definition \ref{defn:unitarily-suspeded}.
\end{proof}

\section{Strongly self-absorbing models}

In this section we show that the stationary inductive limit associated to a unitarily suspended $\Star$-homomorphism is strongly self-absorbing (Theorem \ref{th:SSA}). We end with some general observations, for example that these limits are all isomorphic and define the initial object in the category of strongly self-absorbing $\cstar$-algebras.\\

\begin{notation}
	\label{notation:factor-embeddings}
	Let $n \in \N$ and $\mf p_1, \mf p_2, \cdots, \mf p_n$ be supernatural numbers. Let us define
	$$
		M_{(\mf p_1,\cdots,\mf p_n)} \coloneqq M_{\mf p_1} \otimes M_{\mf p_2} \otimes \cdots \otimes M_{\mf p_n}.
	$$
	If  $i_1,i_2,\cdots,i_k$ are $k$ distinct integers, not necessarily in increasing order,  with $i_k \in \{1,2,\cdots,n\}$, then there is an obvious inclusion 
	$$
		\imath_{(\mf p_1, \cdots, \mf p_n)}^{[i_1,i_2,\cdots,i_k]} \colon M_{(\mf p_{i_1}, \mf p_{i_2}, \cdots, \mf p_{i_k})} \to M_{(\mf p_1,\cdots,\mf p_n)}.
	$$
	Let us furthermore denote 
	$$
		M_{(\mf p_1,\cdots,\mf p_n)}^{[i_1,i_2,\cdots,i_k]} \coloneqq \T{Im} \big (\imath_{(\mf p_1,\cdots,\mf p_n)}^{[i_1,i_2,\cdots,i_k]} \big ).
	$$
	If the $n$-tupel $(\mf p_1, \cdots, \mf p_n)$ in question is clear, we will omit the subscripts and simply write $\imath^{[i_1,\cdots,i_k]}$ and $M^{[i_1,\cdots,i_k]}$.
\end{notation}	

\begin{example}
	The map $\imath_{(\mf p, \mf q)}^{[1]} \colon M_{\mf p} \to M_\mf p \otimes M_\mf q$ for example is the first factor embedding, whereas $\imath_{(\mf p, \mf q)}^{[2]} \colon M_\mf q \to M_\mf p \otimes M_\mf q$ denotes the second factor embedding. Another example is the flip map $M_\mf p \otimes M_\mf p  \to M_\mf p \otimes M_\mf p$,  which is $\imath_{(\mf p, \mf p)}^{[2,1]}$ in our notation.
\end{example}	

\begin{defn}
	\label{def:dim-drop-gen}
	Let $\mf p_1,\cdots, \mf p_n$ and $\mf q_1, \cdots, \mf q_m$ be supernatural numbers.  Let $\mf r \coloneqq (\mf p_1, \mf p_2, \cdots, \mf p_n)$ and $\mf s \coloneqq (\mf q_1, \mf q_2, \cdots, \mf q_m)$. We then define 
	$$
		Z_{\mf r , \mf s} = \left  \{f \in C([0,1], M_{\mf r} \otimes M_{\mf s}) :
			\begin{array}{l}
				f(0) \in M_{\mf r} \otimes 1_\mf s, \\
				f(1) \in 1_\mf r \otimes M_{\mf s}				
			\end{array}
		\right \} .
	$$
\end{defn}

\begin{convention}
	\label{convention:dimension-drop}
	Let $\mf p$ and $\mf q$ be supernatural numbers. To simplify notation a bit, let us define $\mf p^2 \coloneqq (\mf p, \mf p), \ \mf q^2 \coloneqq (\mf q, \mf q)$ and  $\mf r \coloneqq (\mf p^2, \mf q^2, \mf p^2, \mf q^2)$. We then identify $Z_{\mf p^2,\mf q^2} \otimes Z_{\mf p^2,\mf q^2}$ with a $\cstar$-subalgebra of $C([0,1]^2, M_\mf r)$,  where $f \otimes g \in Z_{\mf p^2,\mf q^2} \otimes Z_{\mf p^2,\mf q^2}$ gets identified with the function $(f \otimes g)(s,t) \coloneqq f(s) \otimes g(t)$. Under this identification, one can check that an element $h \in Z_{\mf p^2,\mf q^2} \otimes Z_{\mf p^2,\mf q^2}$  must satisfy:
	\begin{enumerate}[label=$\bullet$]
		\item 	$h(0,0) \in M_\mf r^{[1,2,5,6]}$,
		\item $h(s,0) \in M_\mf r^{[1,2,3,4,5,6]}$,  for $s \in (0,1)$  and 
		\item $h(1,0) \in M_\mf r^{[3,4,5,6]}$.
	\end{enumerate}
	Similar boundary conditions must be fulfilled on the remaining sides of the square.
\end{convention}

\begin{notation}
	\label{notation:mu}
	Let $\mu \colon Z_{\mf p, \mf q} \to Z_{\mf p, \mf q}$ be a unital $\Star$-homomorphism. For $(i,j) \in \{(1,3),(2,4)\}$ we define 
	\[
		\mu^{[i,j]} \colon   Z_{\mf p, \mf q}  \to Z_{\mf p^2, \mf q^2};    \quad  \mu^{[i,j]}_t  \coloneqq  \imath_{(\mf p^2, \mf q^2)}^{[i,j]} \circ \mu_t.
	\]
	Furthermore, we define the following $\Star$-homomorphism:
	\[
		\dbtilde \mu  \colon  Z_{\mf p, \mf q} \otimes Z_{\mf p, \mf q} \to Z_{\mf p^2, \mf q^2};  \quad  \dbtilde \mu_s  (f \otimes g) \coloneqq \mu^{[1,3]}_s(f) \mu^{[2,4]}_s(g) .
	\]
	Note that $\mu_s^{[1,3]}$ and $\mu_s^{[2,4]}$ are commuting $\Star$-homomorphisms.
\end{notation}

\begin{prop}
	\label{prop:intertwining}
	Let $\mf p$ and $\mf q$ be coprime supernatural numbers of infinite type and assume $\mu \colon Z_{\mf p, \mf q} \to Z_{\mf p, \mf q}$ is unitarily suspended. Then 
	$$
	 (\id_{Z_{\mf p^2,\mf q^2}} \otimes 1_{Z_{\mf p^2,\mf q^2}}) \circ \dbtilde \mu \approx_{\T u}  \mu^{[1,3]} \otimes \mu^{[1,3]}.
	$$
	That is, we can make the following diagram commute arbitrary well up to conjugating by a unitary.
	\begin{equation}
		\label{fig:intertwining}
				\begin{aligned}
		\xymatrix{
		Z_{\mf p, \mf q} \otimes Z_{\mf p, \mf q} \ar[rr]^{\mu^{[1,3]}  \otimes \mu^{[1,3]}}  \ar[dr]_{\dbtilde \mu} & & Z_{\mf p^2, \mf q^2} \otimes Z_{\mf p^2, \mf q^2} \\
		& Z_{\mf p^2, \mf q^2}  \ar[ur]_{\id \otimes 1} & 
		}
		\end{aligned}
	\end{equation}
\end{prop}

\begin{proof}
		\textbf{(I)} Let us first do some setup. Recall from Definition \ref{defn:unitarily-suspeded} that  $\mu \colon Z_{\mf p, \mf q} \to Z_{\mf p, \mf q}$ is given by  
	\begin{equation}
		\label{eq:mu}
		\mu_t = \begin{cases}
			\T{ad}(v_t) \circ (\alpha_0 \otimes 1_{\mf q}) & \text{ if } t \in [0,1), \\
			1_{\mf p} \otimes \alpha_1 & \text{ if } t =1,
		\end{cases}
	\end{equation}
	where $(v_t)_{t \in [0,1)}$ is a unitary path in $M_{\mf p} \otimes M_{\mf q}$ starting at the identity and where $\alpha_0 \colon Z_{\mf p, \mf q} \to M_\mf p$ and $\alpha_1\colon Z_{\mf p, \mf q} \to M_\mf q$ are unital $\Star$-homomorphisms. During the proof we identify the codomain of $\mu^{[1,3]} \otimes \mu^{[1,3]}$ with functions on the square, accordingly to Convention \ref{convention:dimension-drop}. Under this identification, the map $\mu^{[1,3]} \otimes \mu^{[1,3]}$ becomes:
	\begin{equation}	
		\label{eq:identification}
		  (\mu^{[1,3]} \otimes \mu^{[1,3]})(f \otimes g)(s,t) = \imath^{[1,3]}_{\mf r }(\mu_s(f)) \cdot \imath^{[5,7]}_{\mf r }(\mu_t(g)) ,
	\end{equation}
	where $\mf r  = (\mf p^2,\mf q^2, \mf p^2, \mf q^2)$ is as in  Convention \ref{convention:dimension-drop}. Finally, we will need to write $Z_{\mf p, \mf q}$ as an inductive limit $\varinjlim (Z_{P_n,Q_n}, \gamma_{n,n+1})$, as explained in Convention \ref{def:dim-drop-limit}.
	\par \noindent \textbf{(II)}  Let us fix  a normalized and finite set $\cF \subseteq Z_{\mf p, \mf q}$ and let $\e > 0$. The goal is to find a unitary $W \in Z_{\mf p^2, \mf q^2} \otimes Z_{\mf p^2, \mf q^2}$ making the diagram  \eqref{fig:intertwining} commute up to $\e$ on  $\cF \otimes \cF \coloneqq \{f \otimes g : f,g \in \cF\}$. For convenience we may assume that the unitary path $(v_t)_{t \in [0,1)}$ is  constant in a small open neighborhood of zero, i.e.~we may assume there exists $0<\eta < 1$ such that $v_t = 1_{\mf p \mf q}$, for $t \in [0,\eta]$. Let us  fix some $0<\delta < \eta  $, such that 
		\begin{equation}
			\label{eq:uniform-continuity}
			\abs{s-s'} \leq \delta \quad \Rightarrow \quad  \norm{\mu_s(f) - \mu_{s'}(f)} \leq \frac \e {10} \qquad (f \in \cF)
		\end{equation}
		and let $n \in \N$ such that
		\begin{equation}
			\label{eq:cond-exp}
			\norm{\kappa_n(\mu(f)) - \mu (f)} \leq \frac \e {10} \qquad (f \in \cF),
		\end{equation}
		where $\kappa_n$ is the conditional expectation as defined in Convention \ref{def:dim-drop-limit}. By possibly enlarging $n$ we can ensure the existence\footnote{This is done by approximating the unitary  $(v_t)_{t \in [0,1-\delta]} \in C([0,1-\delta],M_\mf p \otimes M_\mf q)$ by a unitary in $C([0,1-\delta],M_{P_n} \otimes M_{Q_n})$. } of a unitary path $(\bar v_t)_{t \in [0,1]}$ such that 
		\begin{equation}
			\label{eq:unitary-path-approx}
			\left [
				\begin{array}{ll}
					\bar v_t = 1_{P_n} \otimes 1_{Q_n} & \text{ if } t \in [0,\delta], \\
					\norm{\bar v_t - v_t} \leq \frac \e {10} & \text{ if } t \in [\delta,1-\delta], \\
					\bar v_t = \bar v_{1-\delta} & \text{ if } t \in [1-\delta,1].
				\end{array}
			\right .
		\end{equation}
		Let us denote  $p \coloneqq P_n,\ q \coloneqq Q_n$ and $r \coloneqq (p,p,q,q,p,p,q,q)$.  We then identify $M_r \subseteq M_\mf r$ as unital subalgebra, via our fixed inclusions\footnote{Remember that $M_r$ and $M_\mf r$ are a tensor product of eight matrix respectively UHF-algebras, see Notation \ref{notation:factor-embeddings}.}. Next, let
		\[
			 S_p \in M_p \otimes M_p \quad \text{ and } \quad   S_q \in M_q \otimes M_q
		\]
		denote the self-adjoint flip unitaries. This means $S_p(x \otimes y)S_p^* = y \otimes x$, for all $x,y \in M_p$ and similarly for $M_q$. Then define 
		\begin{equation}
			\label{eq:flip-unitaries}
			\left [
				\begin{array}{l}
					\widetilde S_p \coloneqq\imath_r^{[2,5]}( S_p), \\
					\widetilde S_q \coloneqq \imath_r^{[4,7]}( S_q), \\
					\widetilde S \coloneqq \widetilde S_p \widetilde S_q = \widetilde S_q \widetilde S_p \in M_r^{[2,4,5,7]}
				\end{array}
			\right .
		\end{equation} and for $t \in [0,1]$ we let 
		\begin{equation}
			\label{eq:unitary-path}
				\begin{array}{l}
				u_t \coloneqq \imath_\mf r^{[2,4]}(v_{\min(t,1-\delta)}), \\ 
				 \bar u_t \coloneqq \imath_r^{[2,4]}(\bar v_t).
				 \end{array}
		\end{equation}
		By \eqref{eq:unitary-path-approx} we see that  $\norm{u_t - \bar u_t} \leq \frac \e {10}$, for all $t \in [0,1]$. We are now ready to define a first approximation to the desired unitary $W$:
		\begin{equation}
			\label{eq:U}
			U \colon [0,1]^2 \to M_r : (s,t) \mapsto \widetilde S \bar u_t \bar u_s^*.
		\end{equation}
		Note that $U$ does not yet define an element of $Z_{\mf p^2, \mf q^2} \otimes Z_{\mf p^2, \mf q^2}$. However,  let   us show that the unitary $U$ makes diagram (\ref{fig:intertwining}) commute approximately, i.e. 
		\begin{equation}
			\label{eq:U-aue}
			U (\dbtilde \mu (f \otimes g) \otimes 1_{Z_{\mf p^2, \mf q^2}})U^* \approx_{\frac \e 2} \mu^{[1,3]}(f) \otimes  \mu^{[1,3]}(g)  \qquad (f,g \in \cF).
		\end{equation}
		First, observe the following identities:
		\begin{enumerate}[label=(\alph*)]
			\item $\ad(\widetilde S_p) \circ  \imath_r^{[2,4]} = \imath_r^{[5,4]}$,
			\item $\ad(\widetilde  S_q) \circ \imath_r^{[5,4]} = \imath_r^{[5,7]}$,
			\item $\ad(\widetilde  S_q) \circ \imath_r^{[2,4]} = \imath_r^{[2,7]}$,
			\item $\ad(\widetilde S) \circ \imath_r^{[2,4]} = \imath_r^{[5,7]}$,
			\item $ u_t  u_s^* \mu_s^{[2,4]}(g)  u_s  u_t^* = \mu_t^{[2,4]}(g)$,  for any $g \in Z_{\mf p, \mf q}$ and $s,t \in [0,1-\delta]$.
	\end{enumerate}
	The identities (a)-(d) are easily proven by looking at elementary tensors, whereas (e) follows from the definition of $\mu$.   We are now able to prove (\ref{eq:U-aue}). Let $f,g \in \cF$ and  $(s,t) \in [0,1]^2$. Filling in the definitions gives
		\begin{multline}
		\label{eq:adU}
		  U_{s,t}  (\dbtilde \mu (f \otimes g) \otimes 1_{Z_{\mf p^2, \mf q^2}})(s,t) U_{s,t}^*   
		  \\ =  \widetilde S  \bar u_t  \bar u_s^* \imath^{[1,3]}_\mf r (\mu_s(f))  \imath^{[2,4]}_\mf r (\mu_s(g))   \bar u_s  \bar u_t^* \widetilde S^*.
		 \end{multline}
	Note that  $\bar u_t$ and $\widetilde S$ commute  with $M_\mf r^{[1,3]}$. For $(s,t) \in [0,1-\delta]^2$, for example, one checks that (\ref{eq:adU}) becomes 
	\begin{align*}
		  &\imath_\mf r^{[1,3]}(\mu_s(f)) \widetilde S \bar   u_t \bar u_s^* \imath_\mf r^{[2,4]}(\mu_s(g)) \bar u_s \bar u_t^* \widetilde S^* \\
		\overset{\phantom{()} (\ref{eq:unitary-path-approx})\phantom{(4.4)} }{ \approx_{\frac {2 \e}{10}}}   & \ \imath_\mf r^{[1,3]}(\mu_s(f)) \widetilde S u_tu_s^* \imath_\mf r^{[2,4]}(\mu_s(g)) u_su_t^* \widetilde S^* \\
		 \overset{\phantom{()} \text{(e)} \phantom{(4.5)}  } { =_{\phantom{\frac {2 \e}{10}}}} \   & \  \imath_\mf r^{[1,3]}(\mu_s(f)) \widetilde S \imath_\mf r^{[2,4]}(\mu_t(g)) \widetilde S^* \\
		 \overset{\text{(d)}, (\ref{eq:cond-exp}) \ }{ \approx_{\frac {2 \e}{10}}} \ &  \  \imath_\mf r^{[1,3]}(\mu_s(f)) \imath_\mf r^{[5,7]}(\mu_t(g)) \\
		  \overset{ \phantom{()} (\ref{eq:identification}) \phantom{(4.4)}  }{ =_{\hphantom{\frac {2\e}{10}}}}  & \  (\mu^{[1,3]}(f) \otimes \mu^{[1,3]}(g))(s,t).
	\end{align*}
	For $(s,t) \in ([1-\delta,1] \times [0,1-\delta]) \cup ([0,1-\delta] \times [1-\delta,1]) \cup [1-\delta,1]^2$ one proceeds similarly and uses in particular (\ref{eq:uniform-continuity}). This  proves the claim (\ref{eq:U-aue}).
	\par \noindent \textbf{(III)} The next step is to perturb $U$ on the boundary. Since $S_p$ and $S_q$ are connected to the identity in $M_p \otimes M_p$ respectively $M_q \otimes M_q$ we can find a unitary path $(\nu_t)_{t \in [0,1]}$ inside $M_r$ such that the following holds:
	\begin{equation}
		\label{eq:nu}
		\left [
			\begin{array}{ll}
				\nu_0 = \widetilde S_q, \\
				\nu_t \in \cstar(\widetilde S_q) \subseteq  M_r^{[4,7]} & \text{ if } t \in [0,\delta], \\
				\nu_t = 1_r & \text{ if } t \in [\delta,1-\delta], \\
				\nu_t \in \cstar(\widetilde S_p) \subseteq  M_r^{[2,5]} & \text{ if } t \in [1-\delta,1], \\
				\nu_1 = \widetilde S_p.
			\end{array}
			\right .
	\end{equation}
	Observe that 
	\begin{enumerate}[label=(\alph*)]
		\setcounter{enumi}{5}
		\item $[\widetilde S,\nu_t] = 0$, for every $t \in [0,1]$,
		\item $[\nu_s, \imath_\mf r^{[i,j]}(\mu_s(f))] \approx_{\frac {\e}{10} } 0$, for all $s \in [0,1], \ (i,j) \in \{(2,4),(5,7)\}$ and $f \in \cF$.
	\end{enumerate}
	For (f), let $t \in [0,\delta]$. Then, $\nu_t \in M_r^{[4,7]}$. Since $\widetilde S_p \in M_r^{[2,5]}$ we see that $\widetilde S\nu_t \widetilde S^* = \widetilde S_q \nu_t \widetilde S_q^*$. Note that  $\widetilde S_q$ commutes with $\nu_t$ for $t \in [0,\delta]$, since the $\cstar$-algebra generated by $\widetilde S_q$ is commutative. For $t \in [\delta,1-\delta]$ nothing is to check and for $t \in [1-\delta,1]$ we argue similarly as before. Let us check (g) for $(i,j) = (2,4)$ and $s \in [0,\delta]$:
	$$
		 \nu_s \imath_\mf r^{[2,4]}(\mu_s(f)) \nu_s^* \overset{(\ref{eq:uniform-continuity})} \approx_{\frac \e {10}} \nu_s \imath_\mf r^{[2,4]}(\mu_0(f)) \nu_s^*  = \imath_\mf r^{[2,4]}(\mu_0(f)).
	$$
	The last equality follows from the fact that $\mu_0(f) = \alpha_0(f) \otimes 1_\mf q \in M_\mf p \otimes 1_\mf q$ so that $\imath_\mf r^{[2,4]}(\mu_0(f)) \in M_{\mf r}^{[2]}$, whereas $\nu_s \in M_{\mf r}^{[4,7]}$ for $s \in [0,\delta]$. A similar observation applies to the case $s \in [1-\delta,1]$. We can now define the desired unitary $W$ by 
	$$
		W_{s,t} := \begin{cases}
			\nu_{\min(s,t)}U_{s,t} & \text{ if } (s,t) \in [0,\delta]^2, \\
			\nu_{\max(s,t)}U_{s,t} & \text{ if } (s,t) \in [1-\delta,1]^2,  \\
			\nu_t U_{s,t} \nu_s & \text{ else} .
		\end{cases}
	$$
	This should be compared to the construction of Jiang and Su in \cite[Proposition 8.3]{JS}. Now, we have to check the following:
	\begin{enumerate}[label=(\roman*)]
		\item 	$W$ is continuous,
		\item $W \in Z_{\mf p^2, \mf q^2} \otimes Z_{\mf p^2, \mf q^2}$, 
		\item $\ad(W) \circ (\dbtilde \mu  \otimes 1_{Z_{\mf p^2, \mf q^2}}) \approx_{(\cF \otimes \cF, \e)} \mu^{[1,3]} \otimes \mu^{[1,3]}$.
	\end{enumerate}
	Without reference, we will use that
	$$
	 	U_{s,t} = \widetilde S \qquad ((s,t) \in [0,\delta]^2 \cup [1-\delta,1]^2).
	$$
	We note that (i) follows easily from (f). In order to show (ii), we have to  verify that $W$ satisfies all boundary conditions, as explained in Convention \ref{convention:dimension-drop}. The computation is not difficult but tedious, so as an example we  show how to handle the case $t=0$, which splits up into the following three cases:
		\begin{enumerate}[label=$\bullet$]
			\item $s \in [0,\delta]$ : $W_{s,t} = \nu_{\min(s,t)}U_{s,t} = \nu_0 \widetilde S = \widetilde S_q \widetilde S = \widetilde  S_p \in M_r^{[2,5]}$,
			\item $s \in [\delta,1]$ : $W_{s,t} = \nu_t U_{s,t} \nu_s  = \nu_0 U_{s,0} \nu_s  = \widetilde S_q (\widetilde S \bar u_s^*) \nu_s  = \widetilde S_p \bar u_s^* \nu_s$, which lives in $M_r^{[2,5]} M_r^{[2,4]} M_r^{[2,5]} \subseteq M_r^{[2,4,5]}$,
			\item $s = 1$ : $W_{s,t} = \widetilde S_p \bar u_s^* \nu_s =  \widetilde  S_p \bar  u_{1-\delta}^* \widetilde S_p 
			 \in \widetilde S_p M_r^{[2,4]} \widetilde S_p \underset{\text{(a)}} \subseteq M_r^{[4,5]}$.
		\end{enumerate}
	It follows that $W$ is a unitary in $Z_{\mf p^2, \mf q^2} \otimes Z_{\mf p^2, \mf q^2}$ and it remains to prove (iii). For $(s,t) \in [0,\delta]^2$ and $f,g \in \cF$ we compute  
	 \begin{align*}
	 	W_{s,t} (\dbtilde \mu (f \otimes g) \otimes 1_{Z_{\mf p^2,\mf q^2}})(s,t) W_{s,t}^*  &   \overset{\phantom{(f),(4.4)}}{ =_{\phantom{\frac \e {100}}} }   \nu_{\min(s,t)} \widetilde S \imath_\mf r^{[1,3]}(\mu_s(f)) \imath_\mf r^{[2,4]}(\mu_s(g)) \nu_{\min(s,t)}^* \widetilde S^*  \\
	 	&   \overset{\text{(f)},(\ref{eq:nu})}{=_{\phantom{\frac \e {00}}}}    \widetilde S  \imath_\mf r^{[1,3]}(\mu_s(f)) \nu_{\min(s,t)}\imath_\mf r^{[2,4]}(\mu_s(g)) \nu_{\min(s,t)}^*\widetilde S^* \\
	   & \overset{ \  \text{(g)} \quad \  } {\approx_{\frac {\e} {10}}} \ \widetilde S  \imath_\mf r^{[1,3]}(\mu_s(f)) \imath_\mf r^{[2,4]}(\mu_s(g)) \widetilde S^* \\
	   	  & \overset{\phantom{(f),(4.4)}}{ =_{\phantom{\frac \e {100}}} }   U_{s,t} (\dbtilde \mu(f \otimes g) \otimes 1_{Z_{\mf p^2, \mf q^2}})(s,t) U_{s,t}^*  \\
	   	  &  \overset{(\ref{eq:U-aue}) \ \ }{  \approx_{\frac \e 2}}  (\mu^{[1,3]}(f) \otimes \mu^{[1,3]}(g))(s,t).
	\end{align*}
	The computation for   $(s,t) \in [1-\delta,1]^2$ is similar. For the remaining part, we get 
	\begin{align*}
		W_{s,t} (\dbtilde \mu (f \otimes g) \otimes 1_{Z_{\mf p^2,\mf q^2}})(s,t)W_{s,t}^* & \overset{\phantom{(g) \quad }}{=_{\phantom{\frac {2\e} {10}}} } \nu_t U_{s,t} \nu_s \imath_\mf r^{[1,3]}(\mu_s(f))\imath_\mf r^{[2,4]}(\mu_s(g))\nu_s^*U_{s,t}^*\nu_t^* \\
	& \ \overset{\text{(g)} \quad }{	\approx_{\frac {2\e }{10}}} \nu_t U_{s,t}  \imath_\mf r^{[1,3]}(\mu_s(f)) \imath_\mf r^{[2,4]}(\mu_s(g))U_{s,t}^*\nu_t^* \\
		& \overset{(\ref{eq:U-aue}) \ \ }{  \approx_{\frac \e 2}} \nu_t \imath_\mf r^{[1,3]}(\mu_s(f))\imath_\mf r^{[5,7]}(\mu_t(g)) \nu_t^* \\
	 & \ \overset{\text{(g)} \quad }{	 \approx_{\frac {2\e} {10}}} \imath_\mf r^{[1,3]}(\mu_s(f)) \imath_\mf r^{[5,7]}(\mu_t(g)) \\
	 &  \ \overset{(\ref{eq:identification}) \ \ }{ =_{\phantom{\frac {2\e} {10}}}} (\mu^{[1,3]}(f) \otimes \mu^{[1,3]}(g))(s,t).
	\end{align*}
	This finishes the proof.
\end{proof}

\begin{defn}
	\label{def:APQ}
	Let $\mf p$ and $\mf q$ be coprime supernatural numbers of infinite type. By Theorem \ref{th:existence-unitarily-suspended}, there exists a unitarily suspended standard $\Star$-homomorphism $\mu \colon Z_{\mf p, \mf q} \to Z_{\mf p, \mf q}$, which we fix for the rest of this section. Define $$
		\bar \mu \colon Z_{\mf p, \mf q} \to Z_{\mf p, \mf q}; \quad \bar \mu \coloneqq  \Phi \circ \mu^{[1,3]}, 
	$$ 
	where $\Phi \colon Z_{\mf p^2, \mf q^2} \to Z_{\mf p, \mf q}$ is the canonical isomorphism, induced  fiberwise by isomorphisms $M_{\mf p^2} \to M_\mf p$ and $M_{\mf q^2} \to M_\mf q$.  We then define 
	$$
		\ZZ \coloneqq \varinjlim (Z_{\mf p,\mf q}, \bar  \mu).
	$$
	That is, $\ZZ$ is the stationary inductive limit of $Z_{\mf p, \mf q}$ along $\bar \mu$.
\end{defn}

\begin{theorem}
	\label{th:SSA}
		The $\cstar$-algebra $\ZZ$ is strongly self-absorbing. Furthermore, for any sequence $(\varphi_n)_{n=1}^\infty$ of standard $\Star$-endomorphisms of $Z_{\mf p, \mf q}$ we have 
		$$
		\ZZ  \cong	\varinjlim (Z_{\mf p, \mf q}, \varphi_n) .
		$$
\end{theorem}

\begin{proof} 
	Let us consider the following diagram, where $A \coloneqq \varinjlim (Z_{\mf p, \mf q}, \varphi_n)$:
		\[	
		\xymatrix{
		\ar@{}[d]^{\text{(I)}} & 	Z_{\mf p, \mf q}  \otimes Z_{\mf p, \mf q}  \ar[r]^-{(\bar \mu)^{\otimes 2}} \ar[d]_{\Phi \circ \dbtilde{\mu}} & Z_{\mf p, \mf q}  \otimes Z_{\mf p, \mf q}  \ar[d]_{\Phi \circ \dbtilde{\mu}} \ar[r]^{(\bar \mu)^{\otimes 2}}  & Z_{\mf p, \mf q}  \otimes Z_{\mf p, \mf q}  \ar[d]_{\Phi \circ \dbtilde{\mu}}   \ar[r] & \cdots  \ar[r]  & \ZZ \otimes \ZZ \\
	 \ar@{}[d]^{\text{(II)}} & 	Z_{\mf p, \mf q}  \ar[ur]|-{\id \otimes 1} \ar[r]|-{\bar \mu}  & Z_{\mf p, \mf q}   \ar[ur]|-{\id \otimes 1} \ar[r]|-{\bar \mu} & Z_{\mf p, \mf q}   \ar[r]& \cdots \ar[r] & \ZZ \ar@{-->}[u]_\varphi^\cong \\
	& 	Z_{\mf p, \mf q}  \ar@{=}[u] \ar[ur]|-{\varphi_1} \ar[r]_{\varphi_1}  & Z_{\mf p, \mf q}  \ar@{=}[u]  \ar[ur]|-{\varphi_2} \ar[r]_{\varphi_2} & Z_{\mf p, \mf q} \ar@{=}[u]   \ar[r] & \cdots \ar[r] & A \ar@{-->}[u]^\cong
		}
	\]
	Let us first have a look at row (I). By Proposition \ref{prop:intertwining}, we have $(\id \otimes 1) \circ \dbtilde{\mu} \approx_{\T u} \mu^{[1,3]} \otimes \mu^{[1,3]}$. By composing with $\Phi \otimes \Phi$ it follows that 
	$$
		(\id \otimes 1) \circ (\Phi \circ \dbtilde {\mu}) \approx_{\T u} \bar \mu \otimes \bar \mu.
	$$  
	Furthermore, one  easily computes that 
	$$
		(\Phi \circ \dbtilde{\mu}) \circ (\id \otimes 1) = \Phi \circ \mu^{[1,3]} = \bar \mu.
	$$
	This shows that each triangle in (I) commutes approximately up to conjugating by a unitary.  By \cite[Corollary 2.3.3]{R} it follows that there exist unitaries $(u_n)_{n=1}^\infty$ in $Z_{\mf p, \mf q} \otimes Z_{\mf p, \mf q}$ such that if $\id \otimes 1$ is replaced by $\ad(u_n) \circ (\id \otimes 1)$, then the above diagram is an  approximate intertwining (in the sense of  \cite[Definition 2.3.1]{R}). By the same corollary the induced isomorphism $\varphi \colon \ZZ \to \ZZ \otimes \ZZ$ satisfies 
	$$
		\varphi(x) = \lim_{n \to \infty} u_n (x \otimes 1)u_n^* \qquad (x \in  \ZZ) ,
	$$
	where now each $u_n$ is considered as an element of $\ZZ \otimes \ZZ$. By definition, it follows that $\ZZ$ is strongly self-absorbing.
	\par Next, we look at row (II). Clearly, all lower triangles commute and by Theorem \ref{th:aue-endo} we have that $\varphi_n \approx_{\T u} \bar \mu$, for each $n \in \N$\footnote{Note that $\bar \mu$ is still standard.}. Again, by an approximate intertwining argument, it follows that $A \cong \ZZ$.
\end{proof}

\begin{remark}
The previous theorem picks up the spirit of \cite[Theorem 3.4]{RW} and shows that any stationary inductive limit with a standard (hence trace collapsing) $\Star$-endomorphism is strongly self-absorbing. In our approach however, we do not have to compare the limit to the Jiang-Su algebra $\cZ$.
\end{remark}

\begin{lemma}
	\label{lem:dimension-drop-embedding}
	Let $n \in \N$. Then, there exists $k \in \N$ and a unital $\Star$-homomorphism $Z_{n,n+1} \to Z_{2^k,2^k+1}$.
\end{lemma}

\begin{proof}
	Choose $k \in \N$ such that $2^k \geq n(n+1)$ and write $2^k = n d + r$,  where $r \in \{0,1,\cdots,n-1\}$. Then clearly $d \geq  n +1  > r+1$. By \cite[Proposition 2.1]{SATO}, there exist elements $s,c_1,c_2, \cdots, c_{2^k}$ in $Z_{2^k,2^k+1}$ such that 
	 $$
	 	c_1 \geq 0, \quad c_ic_j^* = \delta_{i,j} c_1^2,  \quad s^*s + \sum_{j=1}^{2^k} c_j^*c_j = 1 \  \text{ and } \ c_1s = s.
	 $$
	 For $i = 1,2, \cdots, n$ define  
	 $$
	 	x_i \coloneqq \sum_{j = (i-1)\cdot d +1 }^{i \cdot d} c_j  \quad \text{ and } \quad   b_i \coloneqq x_i^*x_i = \sum_{j = (i-1) \cdot d +1}^{i \cdot d} c_j^*c_j.
	 $$
	  Note that the $b_i$ are mutually equivalent orthogonal positive elements. By the relations on the $c_j$ and $s$ it is easy to see that  
	 \begin{enumerate}[label=(\alph*)]
	 		\item $1 = b_1 + b_2 + \cdots b_n + (c_{nd +1}^*c_{nd+1} + \cdots + c_{nd+r}^*c_{nd+r}) + s^*s$,
	 		\item $[s^*s] \leq [c_1]$ and $[c_1] = [c_1^2] = [c_1^*c_1]$,
	 		\item $ d[c_1]  = [b_1]$.
	 \end{enumerate}
	 Here $[a]$ denotes the class of an element $a$ in $W(Z_{2^k,2^k+1})$, the Cuntz semigroup of $Z_{2^k,2^k+1}$. Now we have 
	\begin{align*}
		[1-(b_1+\cdots + b_n)] &  \overset{(a)} =  [(c_{nd +1}^*c_{nd+1} + \cdots + c_{nd+r}^*c_{nd+r}) + s^*s] \\
		& \leq [c_{nd+1}^*c_{nd+1}] + \cdots + [c_{nd+r}^*c_{nd+r}] + [s^*s] \\
		& = r [c_1^*c_1] + [s^*s]  \overset{(b)} \leq r[c_1^*c_1] + [c_1] = (r+1)[c_1] \\
		& \ll d [c_1] = [b_1].
	\end{align*}
	The last line follows since $r+1 < d$. Here $\ll$ denotes the \textit{way-below} relation, see for example \cite[Section 3]{NAWATA}. Hence, there exists $\e > 0$ such that 
	$$
		1-(b_1 + \cdots + b_n) \lesssim (b_1 - \e )_+.
	$$		
	 By \cite[Proposition 5.1 (ii)]{RW}, it follows that  $Z_{n,n+1}$ embeds unitally into $Z_{2^k,2^k+1}$.
\end{proof}

\begin{prop}
	\label{prop:dimension-drop-embedding-ssa}
	Let $p,q \in \N$ be coprime. Then $Z_{p,q}$ maps unitally into any strongly self-absorbing $\cstar$-algebra.
\end{prop}

\begin{proof}
	Let $p,q \in \N$ be coprime. We can find $k,l \in \N$ such that $lq - kp = 1$ and $2l > p, \ 2k > q$. With $n \coloneqq kp$, there exists by \cite[Proposition 2.5]{JS} a unital $\Star$-homomorphism $Z_{p,q} \to Z_{kp,lq} = Z_{n,n+1}$ and  by Lemma \ref{lem:dimension-drop-embedding}, $Z_{n,n+1}$ embeds unitally into $Z_{2^k,2^k+1}$, for some large enough $k$. Finally, by \cite[Theorem 3.1]{WSSA}, $Z_{2^k,2^k+1}$ maps unitally into any strongly self-absorbing $\cstar$-algebra.
\end{proof}

\begin{lemma}
	\label{lem:initial-ssa}
	Let $A$ be a strongly self-absorbing $\cstar$-algebra that is locally approximated by prime dimension drop algebras. Then $A$ embeds unitally into any strongly self-absorbing $\cstar$-algebra.
\end{lemma}

\begin{proof}
	 Let $D$ be a strongly self-absorbing $\cstar$-algebra. Since $A$ is locally approximated by prime dimension drop algebras $Z_{p,q}$ with $p,q \in \N$, it follows by \cite[Theorem 3.8]{LOR} that $A$ is an inductive limit of such dimension drop algebras. By Proposition \ref{prop:dimension-drop-embedding-ssa}, each of these embed unitally into $D$ and hence into the central sequence algebra $D_\omega \cap D'$, see \cite[Theorem 2.2]{TW}. By \cite[Proposition 2.2]{TWASH}\footnote{The assumption of $\K_1$-injectivity is this proposition is superfluous.} we see that $D \cong A \otimes D$ and hence $A$ embeds unitally into $D$.
\end{proof}

For the next theorem  we first recall Kirchberg's definition of a  central sequence algebra (cf.~\cite[Definition 1.1]{K}).

\begin{defn}
	Let $A$ be a $\cstar$-algebra. We then define 
	$$
		F(A) \coloneqq \frac {A_\omega \cap A'}{\T{Ann}(A,A_\omega)},
	$$
	where $\T{Ann}(A,A_\omega) \coloneqq \{x \in A_\omega : xA = Ax = \{0\} \}$. 
\end{defn}	

\begin{remark}
	If $A$ is $\sigma$-unital and $(h_n)_{n=1}^\infty$ is an approximate unit for $A$, then $F(A)$ is unital with unit $[(h_n)_{n=1}^\infty] + \T{Ann}(A,A_\omega)$.
\end{remark}

\begin{theorem}
	\label{thm:structure}
		Let $\mf p$ and $\mf q$ be coprime supernatural numbers of infinite type and let $\{\varphi_n \colon Z_{\mf p, \mf q} \to Z_{\mf p, \mf q}\}_{n=1}^\infty$ be a sequence of standard $\Star$-endomorphisms. Let us denote 
		$$
			A \coloneqq \varinjlim (Z_{\mf p, \mf q}, \varphi_n).
		$$
		Then the following holds:
		\begin{enumerate}[label=$\mathrm{(\roman*)}$]
			\item $A$ is the initial object in the category\footnote{The morphisms in this category are approximate unitary equivalence classes of $\Star$-homomorphisms. This ensures that the initial object is unique, cf.~\cite[Proposition 5.12]{TW}.} of strongly self-absorbing $\cstar$-algebras.
			\item If $B$ is any separable $\cstar$-algebra such that  there exists a unital \\ $\Star$-homomorphism $Z_{2,3} \to F(B)$,  then $B \cong B \otimes A$,
			\item If $B$ is any separable $\cstar$-algebra and $Z_{2,3}$ maps unitally into $M(B)_\omega \cap B'$, then $B \cong B \otimes A$, and where $M(B)$ denotes the multiplier algebra of $B$.
	\end{enumerate}
\end{theorem}

\begin{proof}
	(i): By Theorem \ref{th:SSA} it follows that $A$ is strongly self-absorbing. Then Lemma \ref{lem:initial-ssa} shows that $A$ is the initial object in the category of strongly self-absorbing $\cstar$-algebras. (ii): By \cite[Corollary 1.13]{K}, there is a unital \\ $\Star$-homomorphism $Z_{2,3}^{\otimes \infty} \to F(B)$ and hence any dimension drop algebra $Z_{2^n,3^n} \subseteq Z_{2,3}^{\otimes n}$ maps unitally into $F(B)$.  Let $(\mf r, \mf s) \coloneqq (2^\infty, 3^\infty)$. By  a similar argument as in \cite[Proposition 2.2]{TWASH}, it follows that $\overset \longrightarrow {Z_{\mf r, \mf s}^\mu}$ (see Definition \ref{def:APQ}) maps unitally into $F(B)$. By \cite{J}, the algebra $\overset \longrightarrow {Z_{\mf r, \mf s}^\mu}$ is $\K_1$-injective\footnote{The result there is stated for the Jiang-Su algebra $\cZ$, however it only uses that $\cZ$ is a simple inductive limit of prime dimension drop algebras.}. Then the same proof as in \cite[Proposition 5.1]{NAWATA} applies to show that $B \cong B \otimes \overset \longrightarrow {Z_{\mf r, \mf s}^\mu}$. See also \cite[Proposition 4.11 (1)]{K}. By (i) we know that $A \cong \overset \longrightarrow {Z_{\mf r, \mf s}^\mu}$ and hence $B \cong B \otimes A$. (iii): Note that there is a unital $\Star$-homomorphism
	  $$
	  	 M(B)_\omega \cap B' \to F(B) : [(x_n)_{n=1}^\infty] \mapsto [(x_nh_n)_{n=1}^\infty] + \T{Ann}(B,B_\omega) ,
	  $$
	  where $(h_n)_{n=1}^\infty$ is an approximate unit for $B$. It follows that $Z_{2,3}$ maps unitally into $F(B)$ and (ii) applies.
\end{proof}

\begin{remark}
	It is now particularly easy to see that any prime dimension drop algebra $Z_{\mf p, \mf q}$ embeds unitally into any $A$, where $A$ is as in Theorem \ref{thm:structure}. Furthermore, it is encoded in the construction, that these embeddings are standard.
\end{remark}

\bibliographystyle{alpha}
\bibliography{Z}

\end{document}